\documentclass[a4paper,10pt]{article}

\usepackage{a4wide}
\usepackage{amsmath,amssymb,amsfonts,amsthm}
\usepackage{mathtools}
\usepackage[font=footnotesize,labelfont=footnotesize]{caption}
\usepackage{subcaption}
\usepackage{units}
\usepackage{authblk}

\newtheorem{theorem}{Theorem}[section]
\newtheorem{lemma}[theorem]{Lemma}
\newtheorem{corollary}[theorem]{Corollary}

\newtheorem{problem}[theorem]{Problem}
\theoremstyle{remark}
\newtheorem{remark}[theorem]{Remark}

\usepackage[numbers,sort&compress]{natbib}

\usepackage{xcolor}
\usepackage[colorlinks=true,citecolor=red!75!green,linkcolor=blue!75!green]{hyperref}

\newcommand{\vecalg}[1]{\boldsymbol{\mathsf{#1}}}
\newcommand{\matalg}[1]{\boldsymbol{\mathsf{#1}}}
\newcommand{\prodV}[2]{({\textstyle{#1}},{\textstyle{#2}})}
\newcommand{\prodS}[2]{\langle{\textstyle{#1}},{\textstyle{#2}}\rangle}

\newcommand{\im}{\imath}
\newcommand{\BV}{\mathbf{V}}
\newcommand{\bd}{\mathbf{d}}

\newcommand{\bn}{\mathbf{n}}
\newcommand{\bp}{\mathbf{p}}
\newcommand{\bq}{\mathbf{q}}
\newcommand{\bu}{\mathbf{u}}
\newcommand{\bv}{\mathbf{v}}
\newcommand{\bx}{\mathbf{x}}
\newcommand{\bzero}{\mathbf{0}}
\newcommand{\bphi}{\boldsymbol{\phi}}
\newcommand{\bnabla}{\boldsymbol{\nabla}}
\newcommand{\CF}{\mathcal{F}}
\newcommand{\CP}{\mathcal{P}}
\newcommand{\CT}{\mathcal{T}}
\newcommand{\TD}{\textup{D}}
\newcommand{\TI}{\textup{I}}
\newcommand{\TN}{\textup{N}}
\newcommand{\TR}{\textup{R}}
\newcommand{\TS}{\textup{S}}
\newcommand{\BCP}{\boldsymbol{\CP}}

\title{A hybridizable discontinuous Galerkin method with characteristic variables for Helmholtz problems\footnote{Published in \emph{Journal of Computational Physics} (doi: \href{https://doi.org/10.1016/j.jcp.2023.112459}{10.1016/j.jcp.2023.112459}). Distributed under \href{https://creativecommons.org/licenses/by/4.0/}{Creative Commons CC-BY 4.0} license.}}

\author[1]{Axel Modave}
\author[2]{Théophile Chaumont-Frelet}

\affil[1]{POEMS, CNRS, Inria, ENSTA Paris, Institut Polytechnique de Paris, 91120 Palaiseau, France, \href{mailto:axel.modave@ensta-paris.fr}{\texttt{axel.modave@ensta-paris.fr}}}
\affil[2]{Universit\'e C\^ote d'Azur, Inria, CNRS, LJAD, 06902 Sophia Antipolis Cedex, France, \href{mailto:theophile.chaumont@inria.fr}{\texttt{theophile.chaumont@inria.fr}}}

\date{}

\begin{document}

\maketitle

\begin{abstract}
A new hybridizable discontinuous Galerkin method, named the CHDG method, is proposed for solving time-harmonic scalar wave propagation problems.
This method relies on a standard discontinuous Galerkin scheme with upwind numerical fluxes and high-order polynomial bases.
Auxiliary unknowns corresponding to characteristic variables are defined at the interface between the elements, and the physical fields are eliminated to obtain a reduced system.
The reduced system can be written as a fixed-point problem that can be solved with stationary iterative schemes.
Numerical results with 2D benchmarks are presented to study the performance of the approach.
Compared to the standard HDG approach, the properties of the reduced system are improved with CHDG, which is more suited for iterative solution procedures.
The condition number of the reduced system is smaller with CHDG than with the standard HDG method.
Iterative solution procedures with CGNR or GMRES required smaller numbers of iterations with CHDG.
\end{abstract}


\section{Introduction}

Discontinuous Galerkin (DG) finite element methods have proven their strength
to address realistic time-harmonic wave propagation problems,
see e.g.~\cite{li2014hybridizable,barucq2021implementation,barucq2023construction,faucher2020adjoint}.
Due to their ability to handle unstructured and possibly non-conforming meshes,
they are very versatile and can provide high-fidelity solutions to problems
with complicated physical and geometrical configurations.
The DG framework also allows for high-order polynomial basis functions,
which limits dispersion errors occurring when considering high-frequency problems
\cite{ainsworth2004discrete,ainsworth_monk_muniz_2006a,melenk_sauter_2011a}.
Besides, since the degrees of freedom (DOFs) of DG methods are only attached to cells,
they can be linearly indexed in memory, which enables efficient implementation
on vectorized computer architectures, including GPUs,
see e.g.~\cite{klockner2009nodal,modave2016gpu,karakus2019gpu}.

Despite their manifest advantages, the main bottleneck of DG methods
(and more generally, of finite element and finite difference methods)
is the numerical solution of the resulting linear system.
Indeed, although the matrix is sparse, it is typically large, ill-conditioned, and indefinite,
see e.g.~\cite{ernst2012difficult}.
Standard algebraic solvers perform poorly for these systems:
direct solvers are prohibitively costly in large 3D applications;
iterative solvers require less memory storage and allow direct parallel implementations,
but the convergence of the iterative processes can be slow because of intrinsic properties of
the time-harmonic wave propagation problems.
Although preconditioning strategies
have been proposed to speed up the convergence of iterative procedures
and to reduce the computational cost,
see e.g.~\cite{erlangga2004class,farhat2009domain,huber2014hybrid,gander2019class,taus2020sweeps,bootland2021comparison,
gander2022schwarz,royer2022non,dai2022multidirectional},
the development of fast iterative finite element solvers
for high-frequency wave propagation problems remains an active research area.

In this work, we focus on a DG scheme for the Helmholtz equation
in first-order form with upwind fluxes, see e.g.~\cite{hesthaven2007nodal,karniadakis2005spectral}.
Although this approach is very popular in the time-domain,
its direct use for time-harmonic problems is limited,
since it involves many coupled DOFs.
In order to reduce the computational cost, hybridization strategies have been introduced
in the seminal work \cite{cockburn2009unified}, and largely studied over the past decade,
see e.g.~\cite{nguyen2011high,%
griesmaier2011error,%
chen2013hybridizable,%
huerta2013efficiency,%
giorgiani2013hybridizable,%
gopalakrishnan2015stabilization,%
li2013numerical}.
In the resulting hybridizable discontinuous Galerkin (HDG) methods, an additional
``hybrid variable'' corresponding to the Dirichlet trace of the solution is introduced.
This additional variable acts as a Lagrange multiplier that decouples the physical unknowns.
After inverting element-wise local matrices, a reduced system involving
only the Lagrange multiplier is formulated over the skeleton of the mesh.
When using a direct linear solver, the advantage of this approach is straightforward,
as the reduced HDG system features far less DOFs than the original DG system
while preserving its sparsity pattern.
On the other hand, the situation is not as clear when considering iterative solvers,
since the size and filling of the matrix are no longer the main performance criteria.

Here, we propose a novel hybridization strategy in order to accelerate
the solution of the large-scale linear system arising from the upwind DG discretization
of time-harmonic problems with iterative procedures.
This strategy, which we call the CHDG method,
uses the characteristic variables defined at the interface between the elements
as the hybrid variables, as opposed to the Dirichlet traces in the standard HDG
method. This alternative choice of hybrid variable leads to favorable
properties for the resulting reduced system and to more efficient iterative solution procedures
in comparison with the standard hybridization strategy. Specifically,
the reduced system can be written in the form
\begin{equation}
  \label{eqn:fixedPoint}
  (\TI-\Pi\TS)g=b,
\end{equation}
where $g$ corresponds to the characteristic variables, $\Pi$ is an exchange operator
swapping the variables at the interfaces, and $\TS$ is a scattering operator related to
the solution of local element-wise problems. The iteration operator $\Pi\TS$
is a strict contraction, so that the system is well-posed and can be solved with a
simple fixed-point iteration.

Interestingly, the form of the reduced system~\eqref{eqn:fixedPoint} closely resembles
the ultra-weak variational formulation (UWVF) employed in Trefftz discretizations
of time-harmonic problem~\cite{cessenat1998application}.
In fact, our reduced system inherits many of the favorable properties of UWVF matrices.
The advantage of our approach though, is that it simply relies on polynomial basis functions
instead of local solutions.
As a result, volume right-hand sides and heterogeneous media can be readily considered
\cite{hesthaven2007nodal}. Besides, the mesh can be refined, and the discretization order
increased without the conditioning issues typically appearing for plane wave basis functions,
see e.g.~\cite{huttunen2002computational,gabard2007discontinuous,barucq2021local,pernet2022ultra,parolin2022stable}.
To avoid these issues, quasi-Trefftz methods with polynomial basis functions are currently investigated, see e.g.~\cite{imbert2023three}.
The UWVF has been tested with polynomial basis functions in \cite{fure2020discontinuous,monk2010hybridizing}.

The fixed-point system~\eqref{eqn:fixedPoint} also naturally appears
in non-overlapping substructuring domain decomposition (DD) methods.
The iteration operator $\Pi\TS$ was already used
in the seminal work of Despr\'es~\cite{despres1991methodes}.
This formalism and the analogy with a fixed-point system have been widely used,
e.g.~in~\cite{nataf1994optimal,collino2000domain,gander2002optimized,
royer2022non,collino2020exponentially,boubendir2012quasi,modave2020non}.
Our CHDG method can in fact be seen as an element-wise DD method.
The key novelty of our approach, however, is that our
discrete transmission conditions
are built from the numerical fluxes naturally arising in the DG setting.
In particular, cross-points where several mesh faces meet are naturally handled
without any specific treatment.
In contrast, standard DD algorithms based on conforming finite elements require specific
(and sometimes non-local) swap operators to properly account for such cross-points
\cite{claeys2021non,claeys2022robust,pechstein2023unified}.

In this work, the CHDG method with auxiliary characteristic variables is introduced and studied for the numerical solution of Helmholtz problems.
We rigorously show that the resulting reduced system set on the skeleton of the mesh is
well-posed and algebraically equivalent to the original upwind DG method.
Moreover, we prove that this reduced system corresponds to a fixed-point problem
with a strict contraction, which can therefore always be solved with the Richardson iteration.
Then, the performance of CHDG is compared
to the original DG scheme and its standard HDG reformulation
with a sequence of numerical benchmarks.
These examples show that the standard
Richardson iteration always converges without relaxation (although sometimes slowly)
for the CHDG approach, whereas this approach fails to converge for DG and HDG.
Finally, the convergence of standard Krylov methods is compared for the three approaches.
We find that CHDG always requires fewer iterations than DG and HDG
to reach a given accuracy with the GMRES and CGNR iterations.

The remainder of this work is structured as follows.
In Section~\ref{sec:methods}, we introduce the notations, and describe the upwind DG,
the standard HDG, and the CHDG methods as well as their basic properties.
In Section~\ref{sec:redSys}, the reduced system obtained with CHDG is analyzed in detail.
We describe our numerical benchmarks in Section~\ref{sec:linSys}, where we also comment
on the required memory space and conditioning properties of the different approaches.
We study the convergence of standard iterative schemes in Section~\ref{sec:iterProc}
and present our concluding remarks in Section~\ref{sec:conclu}.


\section{Hybridizable discontinuous Galerkin methods}
\label{sec:methods}

Let $\Omega\subset\mathbb{R}^d$, with $d=2$ or $3$, be a Lipschitz polytopal domain.
The boundary $\partial \Omega$ of the domain is partitioned into three non-overlapping
polytopal Lipschitz subsets $\Gamma_{\TD}$, $\Gamma_{\TN}$ and $\Gamma_{\TR}$.
We consider the following time-harmonic scalar wave propagation
problem:
\begin{align}
  \left\{
    \begin{aligned}
      -\im \kappa u + \bnabla\cdot \bq &= 0, && \text{in $\Omega$}, \\
      -\im \kappa \bq + \bnabla u &= \bzero, && \text{in $\Omega$}, \\
      u &= s_{\TD}, && \text{on $\Gamma_\TD$}, \\
      \bn\cdot\bq &= s_{\TN}, && \text{on $\Gamma_\TN$}, \\
      u - \bn\cdot\bq &= s_{\TR}, && \text{on $\Gamma_\TR$},
    \end{aligned}
  \right.
  \label{eqn:pbm}
\end{align}
where the unknowns $u: \Omega \to \mathbb C$ and $\bq: \Omega \to \mathbb C^d$ represent a
time-harmonic wave, $\kappa > 0$ is a given real constant called the wavenumber, and $\bn$
stands for the unit outward normal to $\Omega$. The functions
$s_{\TD}: \Gamma_{\TD} \to \mathbb C$,
$s_{\TN}: \Gamma_{\TN} \to \mathbb C$ and
$s_{\TR}: \Gamma_{\TR} \to \mathbb C$
are boundary data representing an incident field.
Specifically, \eqref{eqn:pbm} is a particular case of the acoustic wave equation,
where we have assumed a time dependence $e^{-\im\omega t}$ for the data and the
solution and $\kappa:=\omega/c$, where $\omega$ is the angular frequency, $t$ is the
time and $c$ is the (uniform) wave speed.
For the sake of brevity, we do not consider volume right-hand sides in the
two first equations of \eqref{eqn:pbm}, but these could be included without difficulty.


\subsection{Mesh, approximation spaces and inner products}
\label{sec:form:notations}

We consider a conforming mesh $\CT_h$ of the domain $\Omega$ consisting of simplicial
elements $K$.
The collection of element boundaries is denoted by
$\partial\CT_h := \{\partial K \:|\: K\in\CT_h\}$,
and the collection of faces is denoted by $\CF_h$.
The collection of faces of an element $K$ is denoted by $\CF_K$.

The approximate fields produced by DG schemes are piecewise polynomials.
Here, for the sake of simplicity, we fix a polynomial degree $p \geq 0$ and introduce
\begin{align*}
  V_h := \prod_{K\in\CT_h} \CP_p(K)
  \qquad \text{and} \qquad
  \BV_h := \prod_{K\in\CT_h} \BCP_p(K),
\end{align*}
where $\CP_p(\cdot)$ and $\BCP_p(\cdot)$ denote spaces of scalar and vector complex-valued
polynomials of degree smaller or equal to $p$. By convention, the restrictions of $u_h\in V_h$
and $\bu_h\in \BV_h$ on $K$ are denoted $u_K$ and $\bu_K$, respectively.

We introduce the sesquilinear forms
\begin{align*}
  \prodV{u}{v}_{K} &:= \int_K u\overline{v}\:d\bx, &
  \prodV{\bu}{\bv}_{K} &:= \int_K \bu\cdot\overline{\bv}\:d\bx, &
  \prodS{u}{v}_{\partial K} &:= \sum_{F\in\CF_K}\int_F u\overline{v}\:d\sigma(\bx), \\
  \prodV{u}{v}_{\CT_h} &:= \sum_{K\in\CT_h} \prodV{u}{v}_{K}, &
  \prodV{\bu}{\bv}_{\CT_h} &:= \sum_{K\in\CT_h} \prodV{\bu}{\bv}_{K}, &
  \prodS{u}{v}_{\partial\CT_h} &:= \sum_{K\in\CT_h} \prodS{u}{v}_{\partial K}.
\end{align*}
By convention, the quantities used in the surface integral $\prodS{\cdot}{\cdot}_{\partial K}$
correspond to the restriction of fields defined on $K$ (e.g.~$v_K$ and $\bv_K$) or quantities
associated to the faces of $K$ (e.g.~$\bn_{K,F}$ with $F\in\CF_K$).


\subsection{Standard DG formulation and numerical fluxes}
\label{sec:form:stdDG}

The general DG formulation of system~\eqref{eqn:pbm} reads:
\begin{problem}
\label{pbm:standardDG}
Find $(u_h,\bq_h) \in V_h\times\BV_h$ such that, for all $(v_h,\bp_h) \in V_h\times\BV_h$,
\begin{align*}
\left\{
\begin{aligned}
  -\im \kappa \prodV{u_h}{v_h}_{\CT_h}
    - \prodV{\bq_h}{\bnabla v_h}_{\CT_h}
    + \prodS{\bn\cdot\widehat{\bq}(u_h,\bq_h)}{v_h}_{\partial\CT_h}
    &= 0, \\
  -\im \kappa \prodV{\bq_h}{\bp_h}_{\CT_h}
    - \prodV{u_h}{\bnabla\cdot \bp_h}_{\CT_h}
    + \prodS{\widehat{u}(u_h,\bq_h)}{\bn\cdot\bp_h}_{\partial\CT_h}
    &= 0,
\end{aligned}
\right.
\end{align*}
where the \emph{numerical fluxes} $\widehat{u}(u_h,\bq_h)$ and $\bn\cdot\widehat{\bq}(u_h,\bq_h)$ are defined face by face below.
\end{problem}

The properties of DG formulations intrinsically depend on the choice of the numerical fluxes.
In this work, we consider \emph{upwind fluxes}. For an interior face $F\not\subset\partial\Omega$
of an element $K$, these fluxes can be written as
\begin{subequations}
\begin{align}
  &\left\{
  \begin{aligned}
    \widehat{u}_F &:= \frac{u_K + u_{K'}}{2} + \bn_{K,F}\cdot\left(\frac{\bq_K-\bq_{K'}}{2}\right), \\
    \bn_{K,F}\cdot\widehat{\bq}_F &:= \bn_{K,F}\cdot\left(\frac{\bq_K+\bq_{K'}}{2}\right) + \frac{u_K - u_{K'}}{2},
  \end{aligned}
  \right.
  \label{eqn:DGflux:int}
\end{align}
where $K'$ is the neighboring element and $\bn_{K,F}$ is the unit outward normal to $K$ on $F$.
For a boundary face $F\subset\partial\Omega$ of an element $K$, the fluxes are defined as
\begin{align}
  &\left\{
  \begin{aligned}
    \widehat{u}_F &:= s_{\TD}, \\
    \bn_{K,F}\cdot\widehat{\bq}_F &:= \bn_{K,F}\cdot\bq_K + (u_K - s_{\TD}),
  \end{aligned}
  \right.
  &&\text{if $F\subset\Gamma_\TD$,}
  \label{eqn:DGflux:bndD} \\
  &\left\{
  \begin{aligned}
    \widehat{u}_F &:= u_K + (\bn_{K,F}\cdot\bq_K - s_{\TN}), \\
    \bn_{K,F}\cdot\widehat{\bq}_F &:= s_{\TN},
  \end{aligned}
  \right.
  &&\text{if $F\subset\Gamma_\TN$,}
  \label{eqn:DGflux:bndN} \\
  &\left\{
  \begin{aligned}
    \widehat{u}_F &:= (u_K + \bn_{K,F}\cdot\bq_K + s_{\TR})/2, \\
    \bn_{K,F}\cdot\widehat{\bq}_F &:= (u_K + \bn_{K,F}\cdot\bq_K - s_{\TR})/2,
  \end{aligned}
  \right.
  &&\text{if $F\subset\Gamma_\TR$.}
  \label{eqn:DGflux:bndR}
\end{align}
\end{subequations}

The upwind fluxes are consistent, which means that $\widehat{u}(u,\bq)=u$ and
$\bn\cdot\widehat{\bq}(u,\bq)=\bn\cdot\bq$ on both interior and boundary faces
when $u$ and $\bq$ are the solution of Problem~\eqref{eqn:pbm}. Under standard
assumptions, the method achieves the optimal convergence rate for the numerical
fields $u_h$ and $\bq_h$ in $L^2$-norm, i.e.~$p+1$ where $p$ is the polynomial degree of
the basis functions. Error estimates have been derived for HDG formulations, equivalent to the
DG formulation above, for the Helmholtz problem with a Dirichlet boundary condition
in~\cite{griesmaier2011error} and a Robin boundary condition
in~\cite{feng2013absolutely,cui2014analysis}. By using a post-processing, the
convergence rate for $u_h$ can be increased by one, see e.g.~\cite{cockburn2008superconvergent}.


\subsection{Hybridization with numerical trace --- Standard HDG method}
\label{sec:form:stdHDG}

In standard HDG formulations, an additional variable $\widehat{u}_h$ corresponding
to the numerical flux $\widehat{u}$ is introduced at the interface between the elements
and on the boundary faces. The discrete unknowns associated to the fields $u_h$ and $\bq_h$
are eliminated in the solution procedure, leading to a reduced system with discrete unknowns
associated to $\widehat{u}_h$ on the skeleton, see e.g.~\cite{cockburn2009unified, cockburn2016static}.

The additional variable, which is called the \emph{numerical trace} in the HDG literature,
belongs to the space $\widehat{V}_h$ defined as
\begin{align*}
  \widehat{V}_h &:= \prod_{F\in\CF_h} \CP_p(F).
\end{align*}
For any field $\widehat{u}_h\in\widehat{V}_h$, there is one set of scalar unknowns associated
to each face of the mesh. After observing that
\begin{align*}
\bn \cdot \widehat{\bq}(u_h,\bq_h)
=
u_h + \bn \cdot \bq_h - \widehat{u}_h,
\end{align*}
we obtain the following HDG formulation, where the numerical trace appears as a hybrid variable:
\begin{problem}
\label{pbm:standardHDG}
Find $(u_h,\bq_h,\widehat{u}_h) \in V_h \times \BV_h \times \widehat{V}_h$ such that, for all $(v_h,\bp_h,\widehat{v}_h) \in V_h \times \BV_h \times \widehat{V}_h$,
\begin{align*}
\left \{
\begin{aligned}
  -\im \kappa \prodV{u_h}{v_h}_{\CT_h}
    - \prodV{\bq_h}{\bnabla v_h}_{\CT_h}
    + \prodS{u_h + \bn\cdot\bq_h -  \widehat{u}_h}{v_h}_{\partial\CT_h}
    &= 0, \\
  -\im \kappa \prodV{\bq_h}{\bp_h}_{\CT_h}
    - \prodV{u_h}{\bnabla\cdot \bp_h}_{\CT_h}
    + \prodS{\widehat{u}_h}{\bn\cdot\bp_h}_{\partial\CT_h}
    &= 0
\end{aligned}
\right.
\end{align*}
and
\begin{multline*}
  \prodS{\widehat{u}_h}{\widehat{v}_h}_{\CF_h}
  - \prodS{\frac{1}{2}(u_h + \bn\cdot\bq_h)}{\widehat{v}_h}_{\partial\CT_h\backslash\partial\Omega}
  - \prodS{u_h + \bn\cdot\bq_h}{\widehat{v}_h}_{\Gamma_\TN}
  - \prodS{\frac{1}{2} (u_h + \bn\cdot\bq_h)}{\widehat{v}_h}_{\Gamma_\TR} \\
  = \prodS{s_{\TD}}{\widehat{v}_h}_{\Gamma_\TD}
  - \prodS{s_{\TN}}{\widehat{v}_h}_{\Gamma_\TN}
  + \prodS{\frac{1}{2} s_{\TR}}{\widehat{v}_h}_{\Gamma_\TR}.
\end{multline*}
\end{problem}

This formulation is equivalent to the standard DG formulation (Problem~\ref{pbm:standardDG})
in the sense that the discrete solutions $u_h$ and $\bq_h$ are identical,
see e.g.~\cite{li2013numerical}.

In the HDG literature \cite{cockburn2009unified,griesmaier2011error,li2013numerical},
a generalization of the above formulation is often considered with
\begin{align*}
\bn\cdot\widehat{\bq}(u_h,\bq_h) = \bn\cdot\bq_h + \tau (u_h - \widehat{u}_h),
\end{align*}
where $\tau$ is the so-called \emph{stabilization function}.
In this work, we focus on the case where $\tau=1$, which corresponds
to the standard upwind fluxes and is widely used in practice.

\begin{remark}[Source projection]
The numerical trace $\widehat{u}_h$ is a polynomial function on every face,
whereas the numerical flux $\widehat{u}$ introduced in the previous section
may be a more general function at any boundary face where the boundary data
does not belong to $\CP_p(F)$. Nevertheless, in practice,
equations~\eqref{eqn:DGflux:bndD}-\eqref{eqn:DGflux:bndR} are still valid
for $\widehat{u}_h$ if the boundary data are projected into the polynomial spaces.
\end{remark}


\subsubsection*{Local element-wise discrete problems}

In the solution procedure, the fields $u_h$ and $\bq_h$ are eliminated by solving local element-wise problems, where the numerical trace $\widehat{u}_h$ is considered as a given data.

For each element $K$, the local problem reads:
\begin{problem}
\label{pbm:HDG:local}
Find $(u_K,\bq_K) \in \CP_p(K)\times\BCP_p(K)$ such that, for all $(v_K,\bp_K) \in \CP_p(K)\times\BCP_p(K)$,
\begin{align*}
\left\{\begin{aligned}
    - \im \kappa \prodV{u_K}{v_K}_{K}
    - \prodV{\bq_K}{\bnabla v_K}_{K}
    + \sum_{F\in\mathcal{F}_K} \prodS{u_K + \bn_{K,F}\cdot\bq_K}{v_K}_{F}
  &=
    \sum_{F\in\mathcal{F}_K}
      \prodS{\widehat{u}_F}{v_K}_{F},
  \\
    - \im \kappa \prodV{\bq_K}{\bp_K}_{K}
    - \prodV{u_K}{\bnabla\cdot \bp_K}_{K}
  &=
    - \sum_{F\in\mathcal{F}_K} \prodS{\widehat{u}_F}{\bn_{K,F}\cdot\bp_K}_{F}, \\
\end{aligned}\right.
\end{align*}
for given surface data {$\widehat{u}_F\in\CP_p(F)$ for all $F\in\CF_K$}.
\end{problem}

This local discrete problem is similar to a Helmholtz problem defined on $K$ with a non-homogeneous Dirichlet boundary condition on $\partial K$.
The discrete problem is well-posed without any condition, as shown e.g.~in \cite{gopalakrishnan2015stabilization}.
We include the proof here for the sake of completeness.

\begin{theorem}[Well-posedness of the local discrete problem]
\label{thm:HDG:local}
Problem~\ref{pbm:HDG:local} is well-posed.
\end{theorem}

\begin{proof}
We simply have to prove that, if $\widehat{u}_F=0$ for all $F\in\CF_K$, the unique solution of Problem~\ref{pbm:HDG:local} is $u_K=0$ and $\bq_K=\bzero$.
For the sake of brevity, the subscripts $K$ and $F$ are omitted for the local fields, the test functions, the unit outgoing normal and the surface data.
Taking both equations of Problem~\ref{pbm:HDG:local} with $v=u$ and $\bp=\bq$ gives
\begin{align*}
  \textstyle
  -\im \kappa \prodV{u}{u}_{K}
    - \prodV{\bq}{\bnabla u}_{K}
    + \prodS{u + \bn\cdot\bq}{u}_{\partial K}
    &= 0, \\
  \textstyle
  -\im \kappa \prodV{\bq}{\bq}_{K}
    - \prodV{u}{\bnabla\cdot\bq}_{K}
    &\textstyle
    = 0.
\end{align*}
Integrating by parts in both equations and taking the complex conjugate lead to
\begin{align*}
  \textstyle
  \im \kappa \prodV{u}{u}_{K}
    + \prodV{u}{\bnabla\cdot\bq}_{K}
    + \prodS{u}{u}_{\partial K}
    &= 0, \\
  \textstyle
  \im \kappa \prodV{\bq}{\bq}_{K}
    + \prodV{\bq}{\bnabla u}_{K}
    - \prodS{\bn\cdot\bq}{u}_{\partial K}
    &\textstyle
    = 0.
\end{align*}
Adding the four previous equations yields $\prodS{u}{u}_{\partial K} = 0$,
and then $u=0$ on $\partial K$. By using this result in Problem~\ref{pbm:HDG:local}, one has
\begin{align*}
\left\{\begin{aligned}
  -\im \kappa \prodV{u}{v}_{K} + \prodV{\bnabla\cdot \bq}{v}_{K} &= 0, \\
  -\im \kappa \prodV{\bq}{\bp}_{K} + \prodV{\bnabla u}{\bp}_{K} &= 0,
\end{aligned}\right.
\end{align*}
for all $[v,\bp] \in \CP_p(K)\times\BCP_p(K)$.
We conclude that
\begin{align*}
  -i\kappa u + \bnabla\cdot \bq &= 0, \\
  -i\kappa \bq + \bnabla u &= \bzero,
\end{align*}
in a strong sense. Because there is no non-trivial polynomial solution
to the previous equations, this yields the result.
\end{proof}

\begin{remark}[Conditioning]
At the continuous level, Helmholtz problems with Dirichlet boundary conditions are ill-posed if the frequency corresponds to an eigenvalue of the Laplace operator.
Here, the Dirichlet conditions are weakly imposed through penalization, so that the discrete
problem are always well-posed. Nevertheless, we shall see in Section \ref{sec:num:condLoc}
that the matrices of the local systems becomes ill-conditioned as $kh$ goes to zero.
\end{remark}


\subsection{Hybridization with characteristic variables --- CHDG method}
\label{sec:form:CHDG}

We propose a new hybridization procedure where the additional variable is associated to incoming and outgoing fluxes at every face of the mesh.
More precisely, the additional variable corresponds to the incoming characteristic variable relative to each element.
Similarly to the standard HDG method, the discrete unknowns associated to the fields $u_h$ and $\bq_h$
are eliminated in the solution procedure, leading to a reduced system with discrete unknowns
associated to the incoming characteristic variable on the skeleton.


\subsubsection*{Characteristic variables}

At each interior face $F\not\subset\partial\Omega$ of an element $K$, the \emph{outgoing characteristic variable $g_{K,F}^{\oplus}$} and the \emph{incoming characteristic variable $g_{K,F}^{\ominus}$} are defined as
\begin{align}
  g_{K,F}^{\oplus} &:= u_K + \bn_{K,F}\cdot\bq_K, \label{eqn:defOutCharac} \\
  g_{K,F}^{\ominus} &:= u_{K'} - \bn_{K,F}\cdot\bq_{K'}, \nonumber
\end{align}
respectively, where $K'$ is the neighboring element.
Let us highlight that the outgoing characteristic variable depends only on values
corresponding to element $K$, whereas the incoming one depends only on values
corresponding to the neighboring element $K'$.
The outgoing characteristic variable of one side corresponds to the incoming one of
the other side, i.e.~$g_{K,F}^{\oplus} = g_{K',F}^{\ominus}$ and
$g_{K,F}^{\ominus} = g_{K',F}^{\oplus}$.
The notations are illustrated on Figure~\ref{fig:notation}.

\begin{figure}[tb!]
  \centering
  \includegraphics[width=0.30\textwidth]{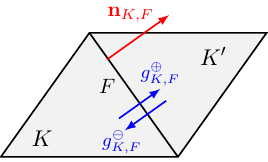}
\caption{Notations for the outgoing and incoming characteristic variables (resp. $g_{K,F}^{\oplus}$ and $g_{K,F}^{\ominus}$) at the face $F$ shared by an element $K$ and a neighboring element $K'$.
Let us note that $g_{K,F}^{\oplus}=g_{K',F}^{\ominus}$ and $g_{K,F}^{\ominus}=g_{K',F}^{\oplus}$.}
\label{fig:notation}
\end{figure}

The characteristic variables can be interpreted as information transported towards
the exterior and the interior of $K$, respectively.
Indeed, let us consider the time-domain version of the governing equations.
Assuming there is no source and the fields are varying only in direction $\bn$,
we get
\begin{align*}
  \left\{\:
    \begin{aligned}
      \partial_t u + c \: \partial_n(\bn\cdot\bq) &= 0, \\
      \partial_t (\bn\cdot\bq) + c \: \partial_n u &= 0.
    \end{aligned}
  \right.
\end{align*}
A simple linear combination gives the transport equations
\begin{align*}
  \left\{\:
    \begin{aligned}
      \partial_t (u + \bn\cdot\bq) + c \: \partial_n (u + \bn\cdot\bq) &= 0, \\
      \partial_t (u-\bn\cdot\bq) - c \: \partial_n (u - \bn\cdot\bq) &= 0.
    \end{aligned}
  \right.
\end{align*}
Therefore, $g^{\oplus}=u+\bn\cdot\bq$ and $g^{\ominus}=u-\bn\cdot\bq$ correspond to quantities transported in the domain in directions $+\bn$ (downstream) and $-\bn$ (upstream), respectively, at velocity $c$.
In the CFD community, the variables $g^{\oplus}$ and $g^{\ominus}$ are generally called \emph{characteristic variables} (see e.g.~\cite{toro2013riemann}), and they are used to define \emph{upwind fluxes} for solving time-dependent problems.
For more general problems, characteristic variables and upwind fluxes are obtained by solving local Riemann problems along the normal direction, see e.g.~\cite{hesthaven2007nodal,toro2013riemann}.

The numerical fluxes \eqref{eqn:DGflux:int} can be rewritten with the characteristic variables as
\begin{align*}
  \left\{
  \begin{aligned}
    \widehat{u}_F &= (g_{K,F}^{\oplus} + g_{K,F}^{\ominus})/2, \\
    \bn_{K,F}\cdot\widehat{\bq}_F &= (g_{K,F}^{\oplus} - g_{K,F}^{\ominus})/2.
  \end{aligned}
  \right.
\end{align*}
If $F$ is a boundary face, i.e.~$F\subset\partial\Omega$, the numerical fluxes and the
outgoing characteristic variable can be defined with~\eqref{eqn:defOutCharac},
but the incoming characteristic variable must be defined differently because there is no
neighboring element. It is defined as
\begin{subequations}
\begin{align}
  g_{K,F}^{\ominus} &:= 2 s_{\TD} - g_{K,F}^{\oplus},
  && \text{if $F\subset\Gamma_\TD$}, \label{eqn:DGchar:bndD} \\
  g_{K,F}^{\ominus} &:= g_{K,F}^{\oplus} - 2s_{\TN},
  && \text{if $F\subset\Gamma_\TN$}, \label{eqn:DGchar:bndN} \\
  g_{K,F}^{\ominus} &:= s_{\TR},
  && \text{if $F\subset\Gamma_\TR$}.\label{eqn:DGchar:bndR}
\end{align}
\end{subequations}
By using these definitions, the numerical fluxes corresponding to the boundary conditions,
i.e.~equations \eqref{eqn:DGflux:bndD}-\eqref{eqn:DGflux:bndR}, are recovered.
Therefore, the boundary conditions are prescribed directly in the definition of the incoming
characteristic variables.


\subsubsection*{CHDG formulation}

In the proposed method, the additional variable, denoted $g^{\ominus}_h$,
corresponds to the incoming characteristic variable at the boundary of all the elements.
The variable $g^{\ominus}_h$ belongs to the space $G_h$ defined as
\begin{align*}
  G_h := \prod_{K\in\CT_h} \prod_{F\in\CF_K} \CP_p(F).
\end{align*}
For any $g^{\ominus}_h\in G_h$, there are two sets of unknowns at each interior
face of the mesh, which correspond to the incoming characteristic variable associated
to the neighboring elements. In the following, the method is called the CHDG method.
The first letter of the name refers to the ``c'' in ``characteristic variable''.

The CHDG formulation reads:
\begin{problem}
\label{pbm:CHDG}
Find $(u_h,\bq_h,g^{\ominus}_h) \in V_h\times\BV_h\times G_h$ such that, for all $(v_h,\bp_h,\xi_h) \in V_h\times\BV_h\times G_h$,
\begin{align*}
\left\{\begin{aligned}
    -\im \kappa \prodV{u_h}{v_h}_{\CT_h}
    - \prodV{\bq_h}{\bnabla v_h}_{\CT_h}
    + \prodS{\frac{1}{2}(g^{\oplus}(u_h,\bq_h)-g^{\ominus}_h)}{v_h}_{\partial\CT_h}
    &= 0,
  \\
    -\im \kappa \prodV{\bq_h}{\bp_h}_{\CT_h}
    - \prodV{u_h}{\bnabla\cdot \bp_h}_{\CT_h}
    + \prodS{\frac{1}{2}(g^{\oplus}(u_h,\bq_h)+g^{\ominus}_h)}{\bn\cdot\bp_h}_{\partial\CT_h}
    &= 0,
\end{aligned}\right.
\end{align*}
and
\begin{align}
  \prodS{g^{\ominus}_h - \Pi(g^{\oplus}(u_h,\bq_h))}{\xi_h}_{\partial\CT_h}
  = \prodS{b}{\xi_h}_{\partial\CT_h},
  \label{eqn:CHDG:addRelation}
\end{align}
with $g^{\oplus}(u_h,\bq_h) := u_h + \bn\cdot\bq_h$.
\end{problem}
The operator $\Pi : G_h \longrightarrow G_h$ used in equation~\eqref{eqn:CHDG:addRelation}
is the \emph{global exchange operator}. It is the key mechanism to enforce the weak
coupling of the element-wise problems at the interior faces and to enforce the boundary
conditions at the boundary faces.
At interior faces, it simply swaps the outgoing characteristics of the two neighboring
elements. This definition is suitably modified
at boundary faces to account for boundary conditions.
For each face $F$ of each
element $K$, $\Pi$ is defined as
\begin{align}
\label{eq_definition_Pi}
\Pi(g^{\oplus})|_{K,F}
=
\left\{
\begin{aligned}
&g^{\oplus}_{K',F} && \text{if $F\not\subset\partial\Omega$ is shared by $K$ and $K'$}, \\
&-g^{\oplus}_{K,F} && \text{if $F\subset\Gamma_\TD$}, \\
&g^{\oplus}_{K,F} && \text{if $F\subset\Gamma_\TN$}, \\
&0 && \text{if $F\subset\Gamma_\TR$},
\end{aligned}
\right.
\end{align}
for any $g^\oplus \in G_h$. For each face $F$ of each element $K$,
the \emph{global right-hand side} $b$ is given by
\begin{align*}
  b|_{K,F}
  &=
  \left\{
  \begin{aligned}
    &0 && \text{if $F\not\subset\partial\Omega$}, \\
    &2s_{\TD} && \text{if $F\subset\Gamma_\TD$}, \\
    &-2s_{\TN} && \text{if $F\subset\Gamma_\TN$}, \\
    &s_{\TR} && \text{if $F\subset\Gamma_\TR$}.
  \end{aligned}
  \right.
\end{align*}
Therefore, Equation \eqref{eqn:CHDG:addRelation} is equivalent to the following relations:
\begin{align*}
    \prodS{g^{\ominus}_{K,F}}{\xi_{K,F}}_{F}
  - \prodS{u_{K'} + \bn_{K',F}\cdot\bq_{K'}}{\xi_{K,F}}_{F}
  &= 0,
  && \text{if $F\not\subset\partial\Omega$,}
  \\
    \prodS{g^{\ominus}_{K,F}}{\xi_{K,F}}_{F}
  + \prodS{u_{K} + \bn_{K,F}\cdot\bq_{K}}{\xi_{K,F}}_{F}
  &= \prodS{2s_{\TD}}{\xi_{K,F}}_{F},
  && \text{if $F\subset\Gamma_\TD$,}
  \\
    \prodS{g^{\ominus}_{K,F}}{\xi_{K,F}}_{F}
  - \prodS{u_{K} + \bn_{K,F}\cdot\bq_{K}}{\xi_{K,F}}_{F}
  &= -\prodS{2s_{\TN}}{\xi_{K,F}}_{F},
  && \text{if $F\subset\Gamma_\TN$,}
  \\
    \prodS{g^{\ominus}_{K,F}}{\xi_{K,F}}_{F}
  &= \prodS{s_{\TR}}{\xi_{K,F}}_{F},
  && \text{if $F\subset\Gamma_\TR$,}
\end{align*}
for each face $F$ of each element $K$. The first relation enforces that
the incoming characteristic variable of an element is the outgoing one of
the neighboring element, and vice versa, for each interior face.
The other relations enforce the boundary conditions.

The CHDG formulation is equivalent to the standard DG formulation (Problem~\ref{pbm:standardDG}),
and thus to the standard HDG formulation (Problem~\ref{pbm:standardHDG}).
Similar to the standard HDG formulation, the additional variable $g^{\ominus}_h$
is a polynomial function on each face,
whereas the incoming characteristic variable introduced previously could be a more general
function on the boundary of the domain. Nevertheless, equations~\eqref{eqn:DGchar:bndD}-\eqref{eqn:DGchar:bndR}
still hold up to projecting the right-hand sides onto piecewise polynomials.


\subsubsection*{Local element-wise discrete problems}

The hybridization procedure leads to a reduced system with discrete unknowns
associated to the incoming characteristic variable $g^{\ominus}_h$
on the skeleton. This elimination is achieved by solving local element-wise problems,
where the incoming characteristic variable is considered as a given data.

For each element $K$, the local problem reads:
\begin{problem}
\label{pbm:CHDG:local}
Find $(u_K,\bq_K) \in \CP_p(K)\times\BCP_p(K)$ such that, for all $(v_K,\bp_K) \in \CP_p(K)\times\BCP_p(K)$,
\begin{align*}
\left\{\begin{aligned}
  -\im \kappa \prodV{u_K}{v_K}_{K}
    - \prodV{\bq_K}{\bnabla v_K}_{K}
    + {\sum_{F\in\mathcal{F}_K}} \prodS{\frac{1}{2} {g_{K,F}^\oplus}}{v_K}_{F}
  &= {\sum_{F\in\mathcal{F}_K}} \prodS{\frac{1}{2} {g^\ominus_{K,F}}}{v_K}_{F}, \\
  -\im \kappa \prodV{\bq_K}{\bp_K}_{K}
    - \prodV{u_K}{\bnabla\cdot \bp_K}_{K}
    + {\sum_{F\in\mathcal{F}_K}} \prodS{\frac{1}{2} {g_{K,F}^\oplus}}{\bn_{K,F}\cdot\bp_K}_{F}
  &= - {\sum_{F\in\mathcal{F}_K}} \prodS{\frac{1}{2} {g^\ominus_{K,F}}}{\bn_{K,F}\cdot\bp_K}_{F},
\end{aligned}\right.
\end{align*}
with $g_{K,F}^\oplus=u_K + \bn_{K,F}\cdot\bq_K$, for given surface data {$g_{K,F}^\ominus\in\CP_p(F)$ for all $F\in\CF_K$}.
\end{problem}

The local problem can be interpreted as a discretized Helmholtz problem defined on $K$ with a non-homogeneous Robin boundary condition on $\partial K$.
We show hereafter that this discrete problem is well-posed.

\begin{theorem}[Well-posedness of the local discrete problem]
Problem~\ref{pbm:CHDG:local} is well-posed.
\end{theorem}
\begin{proof}
We simply have to prove that, if $g_{K,F}^\ominus=0$ for all $F\in\CF_K$, the unique solution
of Problem~\ref{pbm:CHDG:local} is $u_K=0$ and $\bq_K=\bzero$.
For the sake of brevity, the subscripts $K$ and $F$ are omitted for the local fields, the test functions, the unit outgoing normal and the surface data.
Taking both equations of Problem~\eqref{pbm:CHDG:local} with $v=u$ and $\bp=\bq$ gives
\begin{align*}
  -\im \kappa \prodV{u}{u}_{K}
    - \prodV{\bq}{\bnabla u}_{K}
    + \prodS{\frac{1}{2} (u + \bn\cdot\bq)}{u}_{\partial K}
  &= 0, \\
  -\im \kappa \prodV{\bq}{\bq}_{K}
    - \prodV{u}{\bnabla\cdot \bq}_{K}
    + \prodS{\frac{1}{2} (u + \bn\cdot\bq)}{\bn\cdot\bq}_{\partial K}
  &= 0.
\end{align*}
Integrating by parts in both equations and taking the complex conjugate lead to
\begin{align*}
  \im \kappa \prodV{u}{u}_{K}
    + \prodV{u}{\bnabla\cdot \bq}_{K}
    + \prodS{u}{\frac{1}{2} (u - \bn\cdot\bq)}_{\partial K}
  &= 0, \\
  \im \kappa \prodV{\bq}{\bq}_{K}
    + \prodV{\bq}{\bnabla u}_{K}
    - \prodS{\bn\cdot\bq}{\frac{1}{2} (u - \bn\cdot\bq)}_{\partial K}
  &= 0.
\end{align*}
Adding the four previous equations yields $\prodS{u}{u}_{\partial K} + \prodS{\bn\cdot\bq}{\bn\cdot\bq}_{\partial K} = 0$, which gives $u=0$ and $\bn\cdot\bq=0$ on $\partial K$.
By using these boundary conditions in
Problem~\eqref{pbm:CHDG:local},
we have that the fields should be a solution of the strong problem.
Because there is no solution with both homogeneous Neumann and Dirichlet boundary conditions, this yields the result.
\end{proof}

\begin{remark}[Conditioning]
In contrast to Helmholtz problems with Dirichlet boundary conditions,
the local problems with Robin boundary conditions are always well-posed at
the continuous level. We shall see in Section~\ref{sec:num:condLoc}
that the matrices of the local systems stays well-conditioned as $kh$ goes to zero for low-order
finite elements, and that the condition number is smaller than with HDG for high-order finite
elements.
\end{remark}


\section{Analysis of the reduced system for the CHDG method}
\label{sec:redSys}

In this section, we introduce and study the reduced version of the hybridized formulation with characteristic variables (Problem~\ref{pbm:CHDG}).
This version is obtained by solving the local element-wise problems (Problem~\ref{pbm:CHDG:local}) and then eliminating the physical variables $u_h$ and $\bq_h$ from the system.


\subsection{Formulation of the reduced system}

In order to write the problem in a reduced formulation, we introduce the
\emph{global scattering operator} $\TS : G_h \longrightarrow G_h$ defined such that,
for each face $F$ of each element $K$,
\begin{align}
  \left.\TS(g^{\ominus}_h)\right|_{K,F} &:= u_K(g^{\ominus}_h) + \bn_{K,F}\cdot\bq_K(g^{\ominus}_h),
  \label{eqn:defS}
\end{align}
where $(u_K,\bq_K)$ is the solution of Problem~\ref{pbm:CHDG:local} with the incoming characteristic
data $(g^{\ominus}_{K,F})_{F \in \CF_K}$ contained in $g^{\ominus}_h$ as a given surface data.
This operator can be interpreted as an \emph{``incoming characteristic variable to outgoing characteristic variable''} operator.

By using the operator $\TS$, Problem~\ref{pbm:CHDG} is rewritten as:
\begin{problem}
\label{pbm:CHDG:reducedTmp}
Find $g^{\ominus}_h \in G_h$ such that, for all $\xi_h \in G_h$,
\begin{align*}
    \prodS{g^{\ominus}_h}{\xi_h}_{\partial\CT_h}
    - \prodS{\Pi(\TS(g^{\ominus}_h))}{\xi_h}_{\partial\CT_h}
    &= \prodS{b}{\xi_h}_{\partial\CT_h}.
\end{align*}
\end{problem}

In order to write the problem in a more compact form, we introduce the
\emph{global projected right-hand side} $b_h := P_h b \in G_h$, where
$P_h : L^2(\partial\CT_h) \longrightarrow G_h$ is the projection operator
defined such that $\prodS{P_h b}{\xi_h}_{\partial\CT_h} = \prodS{b}{\xi_h}_{\partial\CT_h}$
for all $\xi_h\in G_h$. Problem~\ref{pbm:CHDG:reducedTmp} can then be rewritten as:
\begin{problem}
\label{pbm:CHDG:hybridized}
Find $g^{\ominus}_h \in G_h$ such that
\begin{align*}
  (\TI - \Pi \TS) g^{\ominus}_h = b_h.
\end{align*}
\end{problem}

Problems~\ref{pbm:CHDG:reducedTmp} and~\ref{pbm:CHDG:hybridized}
are equivalent to Problem~\ref{pbm:CHDG} because the element-wise local
problems (Problem~\ref{pbm:CHDG:local}) are well-posed.
As discussed in the introduction, Problem~\ref{pbm:CHDG:hybridized} is similar to
formulations obtained for DD and UWVF methods to solve Helmholtz problems.


\subsection{Fixed-point problem}
\label{sec:fixedpointpbm}

Problem~\ref{pbm:CHDG:hybridized} corresponds to a fixed-point problem.
In this section, we prove that the operator $\Pi \TS$ is a strict contraction.
As a consequence, the fixed-point problem is always well-posed, and it can (at least in principle) be solved with stationary iterative procedures.
The algebraic version of this system is discussed in Sections~\ref{sec:num:discretization} and \ref{sec:iter:Richardson}.

The properties of $\TS$ and $\Pi$ are established by using a norm on $G_h$ defined as
\begin{align*}
    \|g^{\ominus}_h\| := \sqrt{\sum_{K\in\CT_h}\sum_{F\in\mathcal{F}_K} \|g^{\ominus}_{K,F}\|_{F}^2},
\end{align*}
where $\|{\cdot}\|_{F}^2$ is the natural norm of $L^2(F)$.
We start with a technical lemma.

\begin{lemma}
\label{lemma:contractS}
(i) The solution of Problem~\ref{pbm:CHDG:local} verifies
\begin{align}
  \sum_{F\in\mathcal{F}_K} \|u_K + \bn_{K,F}\cdot\bq_K\|^2_{F}
  + \sum_{F\in\mathcal{F}_K} \|u_K - \bn_{K,F}\cdot\bq_K -  {g_{K,F}^\ominus}\|^2_{F}
  = \sum_{F\in\mathcal{F}_K} \| {g_{K,F}^\ominus}\|^2_{F}.
  \label{eqn:lemma:localSys}
\end{align}
(ii) The second term in the left-hand side of \eqref{eqn:lemma:localSys} vanishes
if and only if ${g_{K,F}^\ominus}=0$.
\end{lemma}

\begin{proof}
For the sake of brevity, the subscripts $K$ and $F$ are omitted for the local fields, the test functions, the unit outgoing normal and the surface data.

\emph{(i)} Taking both equations of Problem~\ref{pbm:CHDG:local} with $v=u$ and $\bp=\bq$ gives
\begin{align*}
  -\im \kappa \prodV{u}{u}_{K}
    - \prodV{\bq}{\bnabla u}_{K}
    + \prodS{\frac{1}{2} (u + \bn\cdot\bq)}{u}_{\partial K}
  &= \prodS{\frac{1}{2} g^\ominus}{u}_{\partial K} \\
  -\im \kappa \prodV{\bq}{\bq}_{K}
    - \prodV{u}{\bnabla\cdot \bq}_{K}
    + \prodS{\frac{1}{2} (u + \bn\cdot\bq)}{\bn\cdot\bq}_{\partial K}
  &= - \prodS{\frac{1}{2} g^\ominus}{\bn\cdot\bq}_{\partial K}.
\end{align*}
Integrating by parts in both equations and taking the complex conjugate lead to
\begin{align*}
  \im \kappa \prodV{u}{u}_{K}
    + \prodV{u}{\bnabla\cdot \bq}_{K}
    + \prodS{u}{\frac{1}{2} (u - \bn\cdot\bq)}_{\partial K}
  &= \prodS{u}{\frac{1}{2} g^\ominus}_{\partial K} \\
  \im \kappa \prodV{\bq}{\bq}_{K}
    + \prodV{\bq}{\bnabla u}_{K}
    - \prodS{\bn\cdot\bq}{\frac{1}{2} (u - \bn\cdot\bq)}_{\partial K}
  &= - \prodS{\bn\cdot\bq}{\frac{1}{2} g^\ominus}_{\partial K}.
\end{align*}
Adding the four previous equations yields
\begin{multline*}
  \frac{1}{2}\prodS{(u + \bn\cdot\bq)}{(u + \bn\cdot\bq)}_{\partial K}
  + \frac{1}{2}\prodS{(u - \bn\cdot\bq)}{(u - \bn\cdot\bq)}_{\partial K} \\
  = \frac{1}{2}\prodS{g^\ominus}{(u - \bn\cdot\bq)}_{\partial K}
  + \frac{1}{2}\prodS{(u - \bn\cdot\bq)}{g^\ominus}_{\partial K},
\end{multline*}
and then
\begin{align*}
    \|u + \bn\cdot\bq\|_{\partial K}^2
  + \|u - \bn\cdot\bq\|_{\partial K}^2
  = \|u - \bn\cdot\bq\|_{\partial K}^2
  - \|u - \bn\cdot\bq -  g^\ominus\|_{\partial K}^2
  + \|g^\ominus\|_{\partial K}^2,
\end{align*}
which gives the result~\eqref{eqn:lemma:localSys}.

\emph{(ii)} If the second term in the left-hand side of~\eqref{eqn:lemma:localSys}
vanishes, then $g^\ominus = u - \bn\cdot\bq$ on $\partial K$. Using this relation in
Problem~\ref{pbm:CHDG:local}, we see that $u$ and $\bq$ must satisfy
\begin{align*}
\left \{
\begin{aligned}
  \textstyle
  -\im \kappa \prodV{u}{v}_{K}
    - \prodV{\bq}{\bnabla v}_{K}
    + \prodS{\bn\cdot\bq}{v}_{\partial K}
  &= 0, \\
  \textstyle
  -\im \kappa \prodV{\bq}{\bp}_{K}
    - \prodV{u}{\bnabla\cdot \bp}_{K}
    + \prodS{u}{\bn\cdot\bp}_{\partial K}
  &= 0
\end{aligned}
\right .
\end{align*}
for all $v \in \CP_p(K)$ and $\bp \in \BCP_p(K)$,
and integration by parts shows that $u$ and $\bq$ solve the Helmholtz
equation in strong form. But as we have already seen in the proof of Theorem
\ref{thm:HDG:local}, there is no non-trivial polynomial solution, meaning
that $u=0$ and $\bq=\bzero$, and then $g^\ominus=0$. The converse statement is direct,
because the local problem is well-posed.
\end{proof}

\begin{theorem}
\label{theorem_TS}
The scattering operator $\TS$ is a strict contraction, i.e.
\begin{align*}
  \|\TS(g^{\ominus}_h)\| < \|g^{\ominus}_h\|, \quad \forall g^{\ominus}_h\in G_h\backslash\{0\}.
\end{align*}
\end{theorem}
\begin{proof}
Let $g^{\ominus}_h\in G_h\backslash\{0\}$. By Lemma~\ref{lemma:contractS}, one has
\begin{align*}
  \sum_{F\in\CF_K} \|u_K + \bn_{K,F}\cdot\bq_K\|_{F}^2 < \sum_{F\in\CF_K} \|g^{\ominus}_{K,F}\|_{F}^2.
\end{align*}
The equality cannot happen because $g^{\ominus}_h\neq0$.
Then, by using the definition of $\TS$ (i.e.~equation \eqref{eqn:defS}), one has
\begin{align*}
  \sum_{F\in\CF_K} \|\TS(g^{\ominus}_h)|_{K,F}\|_{F}^2 < \sum_{F\in\CF_K} \|g^{\ominus}_{K,F}\|_{F}^2.
\end{align*}
Summing this estimate over all $K\in\CT_h$ gives the result.
\end{proof}

The global scattering operator $\TS$ is always strictly contracting whereas, in a continuous context, it preserves energy.
The proof of Theorem \ref{theorem_TS} uses the fact that there are no polynomial solution to the Helmholtz equation, and therefore, the strict contraction property of $\TS$ is a numerical artifact that is not physical.
This is related to the fact the upwind DG scheme is a dissipative method to start with \cite{ainsworth_monk_muniz_2006a}.

\begin{theorem}
The exchange operator $\Pi$ is a contraction, i.e.
\begin{align*}
    \|\Pi(g^{\ominus}_h)\| \leq \|g^{\ominus}_h\|, \quad \forall g^{\ominus}_h\in G_h.
\end{align*}
In addition, if $\Gamma_\TR=\emptyset$, $\Pi$ is an involution, i.e.~$\Pi^2=\TI$, and an isometry, i.e.
\begin{align*}
  \|\Pi(g^{\ominus}_h)\| = \|g^{\ominus}_h\|, \quad \forall g^{\ominus}_h\in G_h.
\end{align*}
\end{theorem}
\begin{proof}
These results are straightforward consequences of the definition of $\Pi$.
\end{proof}

As a consequence of the two previous theorems, we have the following result.

\begin{corollary}
\label{corol:contraction}
The operator $\Pi\TS$ is a strict contraction, i.e.
\begin{align*}
    \|\Pi\TS(g^{\ominus}_h)\| < \|g^{\ominus}_h\|, \quad \forall g^{\ominus}_h\in G_h\backslash\{0\}.
\end{align*}
\end{corollary}

The strict contraction property of Corollary \ref{corol:contraction} is due to the fact that $\Pi$ and/or $\TS$ dissipate energy.
Actually, the global scattering operator $\TS$ is always strictly contracting.
As discussed above, this can be related to the fact the upwind DG scheme is a dissipative method.
On the other hand, the global exchange operator $\Pi$ can only dissipate energy in the presence of a Robin boundary (see the last line of \eqref{eq_definition_Pi}), otherwise it is an involution.
Therefore, we may identify two possible sources of dissipation. The first source is numerical dissipation which is always present, but
may become small as the mesh is refined, leading to possibly slow convergence of fixed
point iterations in energy-preserving problem. The other source of dissipation comes
from physical absorption and should lead to faster convergence rates on fine meshes.
The numerical examples we present in Section~\ref{sec:iter:Richardson} clearly depict how
the presence or absence of physical dissipation impact the convergence rates of fixed
point iterations.

Let us note that, for conservative methods (including standard conforming finite elements), where $\TS$ does not dissipate, $\Pi\TS$ should preserve energy if there is no physical dissipation.
In fact, the convergence of standard DD algorithms is proven only for energy-preserving problems with relaxation, e.g.~\cite{claeys2022robust}.
It has been proven recently in \cite{pernet2022ultra} that the iteration matrix of a Trefftz DG method is also a strict contraction for a configuration with a Robin boundary condition.
To the best of our knowledge, this is the only other example of finite element method that can be written with a strictly contracting iterative matrix for Helmholtz problems.


\section{Linear algebraic systems}
\label{sec:linSys}

In this section, the algebraic systems resulting from the DG discretization and its two possible hybridizations are studied for two-dimensional problems.
After a description of the polynomial basis and reference benchmarks in Sections~\ref{sec:num:discretization} and~\ref{sec:num:benchmarks}, respectively, the required memory storage is discussed in Section~\ref{sec:num:storage}.
The condition numbers of the local element-wise matrices and the global reduced matrices are discussed in Sections~\ref{sec:num:condLoc} and~\ref{sec:num:condGlo}, respectively.


\subsection{Polynomial basis functions}
\label{sec:num:discretization}

The physical fields $u_h$ and $\bq_h$ are represented with standard hierarchical shape functions.
These functions are built with tensor products of Lobatto shape functions
(see e.g.~\cite[section 2.2.3]{solin2003higher} and~\cite{beriot2016efficient}).
For triangular elements, they are classified into vertex, edge, and bubble functions.
Since the bubble functions vanish on the edges of the triangle, only the degrees of freedom associated
to vertex and edge functions are involved in the boundary and interface integrals of the
variational formulations. In remainder of this work, the edges of the triangular elements are
called ``faces'' in order to follow the general terminology.

The fields defined on the skeleton, i.e.~$\widehat{u}_h$ for HDG and $g_h^\ominus$ for CHDG,
are univariate polynomials. A possible
choice for the shape functions would be the Lobatto shape functions,
which correspond to the restriction of the shape functions used for the physical fields.
Instead, we consider scaled Legendre shape functions, which are orthogonal
in $L^2(F)$ for each face $F$.
For each element, they are scaled in such a way that the local mass matrices are the identity matrix, \text{i.e.}
\begin{align*}
  \prodV{\phi_i^F}{\phi_j^F}_F = \delta_{ij}, \quad \text{for } i,j=1,\ldots,N_\mathrm{dof\cdot per\cdot fce},
\end{align*}
where $\phi_i^F$ and $\phi_j^F$ are the shape functions associated with the face $F$, and $N_\mathrm{dof\cdot per\cdot fce}$ is the number of degrees of freedom per face.

The Lobatto functions and the scaled Legendre functions give rigorously the same numerical solution (up to floating point errors), as they are two equivalent sets of basis functions, but they lead to different algebraic systems.
Let us consider the algebraic system resulting from the finite element discretization of Problem~\ref{pbm:CHDG:reducedTmp}.
With the Lobatto functions, the first term of this problem corresponds to a mass matrix in the algebraic system.
By contrast, with the scaled Legendre functions, it corresponds to an identity matrix as the shape functions are orthonormal.
In fact, the system corresponding to the scaled Legendre functions, denoted $\matalg{A}\vecalg{g}=\vecalg{b}$, can be obtained from the system corresponding to the Lobatto functions, denoted $\matalg{A}_\mathrm{Lob}\vecalg{g}_\mathrm{Lob}=\vecalg{b}_\mathrm{Lob}$, by using a symmetric preconditioning:
\begin{align}
  \underbrace{\big(\matalg{M}_\mathrm{Lob}^{-\nicefrac{1}{2}}\matalg{A}_\mathrm{Lob}\matalg{M}_\mathrm{Lob}^{-\nicefrac{1}{2}}\big)}_{\displaystyle\matalg{A}}
  \underbrace{\big(\matalg{M}_\mathrm{Lob}^{\nicefrac{1}{2}}\vecalg{g}_\mathrm{Lob}\big)}_{\displaystyle\vecalg{g}}
  =
  \underbrace{\big(\matalg{M}_\mathrm{Lob}^{-\nicefrac{1}{2}}\vecalg{b}_\mathrm{Lob}\big)}_{\displaystyle\vecalg{b}},
  \label{eqn:prec}
\end{align}
where $\matalg{M}_{\mathrm{Lob}}$ is the mass matrix associated to the faces.

In preliminary comparison studies (not shown), we have observed that, for both HDG and CHDG methods, the convergence of the iterative solution procedures (without preconditioning strategy) is faster with the scaled Legendre functions than with the Lobatto functions.
Here is a partial explanation.
With the scaled Legendre functions, the scalar product $(\cdot,\cdot)_F$ of two fields is equal to the algebraic inner product on the corresponding components.
Similarly, the $L^2$-norm of a field is equal to the $2$-norm of its components.
Therefore, the inner product and the norm used in the standard iterative solution procedures are in some sense \emph{``natural''} for the considered problems.
Note that this approach is rigorously equivalent to using the Lobbato functions with a symmetric preconditioning with the mass matrix $\matalg{M}_{\mathrm{Lob}}$, see equation \eqref{eqn:prec}.
In fact, that preconditioning approach is equivalent to using $\matalg{M}_{\mathrm{Lob}}$ as a left preconditioner and using the scalar $(\cdot,\cdot)_F$ as inner product in weighted Krylov methods.

For the sake of brevity, only results with the scaled Legendre functions are presented in the remainder of this article.


\subsection{Reference benchmarks}
\label{sec:num:benchmarks}

To study the properties of the algebraic systems and the convergence of iterative solution procedures, we consider three benchmarks corresponding to different physical configurations, already used in~\cite{chaumont2022controllability}.
Snapshots of the real part of the solutions are shown in Figure~\ref{fig:benchmarks}.
The numerical simulations have been performed with a dedicated \textsf{MATLAB} code.
The mesh generation and the visualization have been done with \textsf{gmsh}~\cite{geuzaine2009gmsh} (version 4.11.1).
In all the cases, third-degree polynomial bases, i.e.~$p=3$, have been used.
The parameter $h$ is the element size provided in \textsf{gmsh}.

\begin{figure}[tb!]
\centering
\begin{subfigure}[b]{0.46\textwidth}
  \centering
  \caption{Benchmark 1 (plane wave)}
  \includegraphics[width=0.99\textwidth]{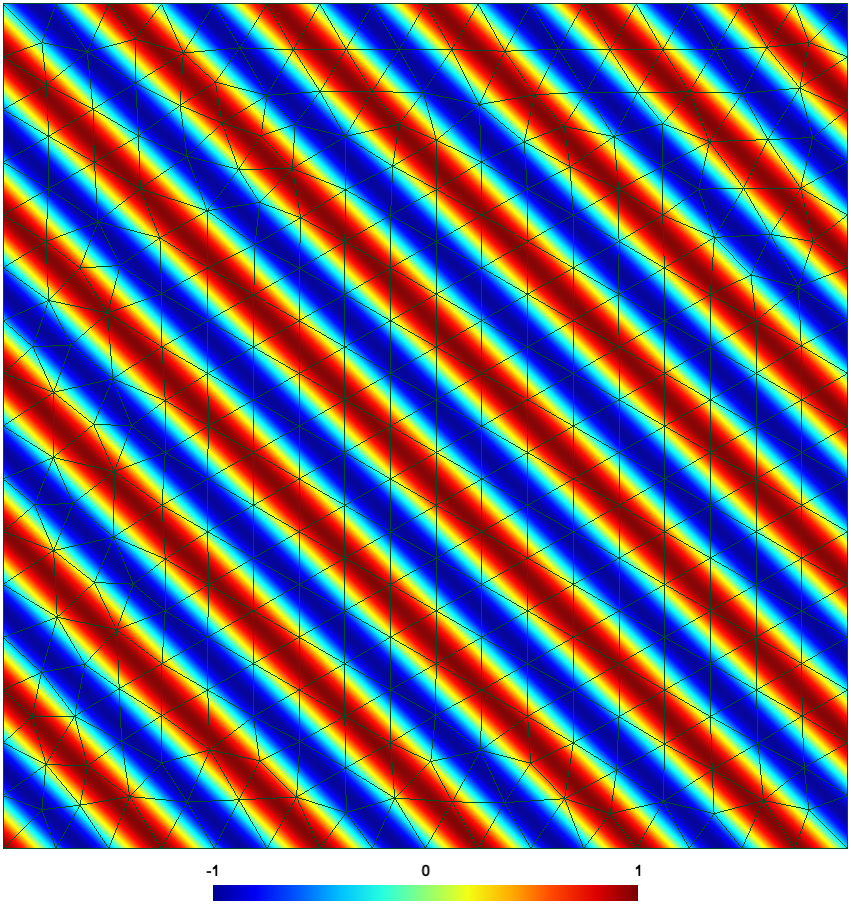}
\end{subfigure}
\qquad
\begin{subfigure}[b]{0.46\textwidth}
  \centering
  \caption{Benchmark 2 (cavity)}
  \includegraphics[width=0.99\textwidth]{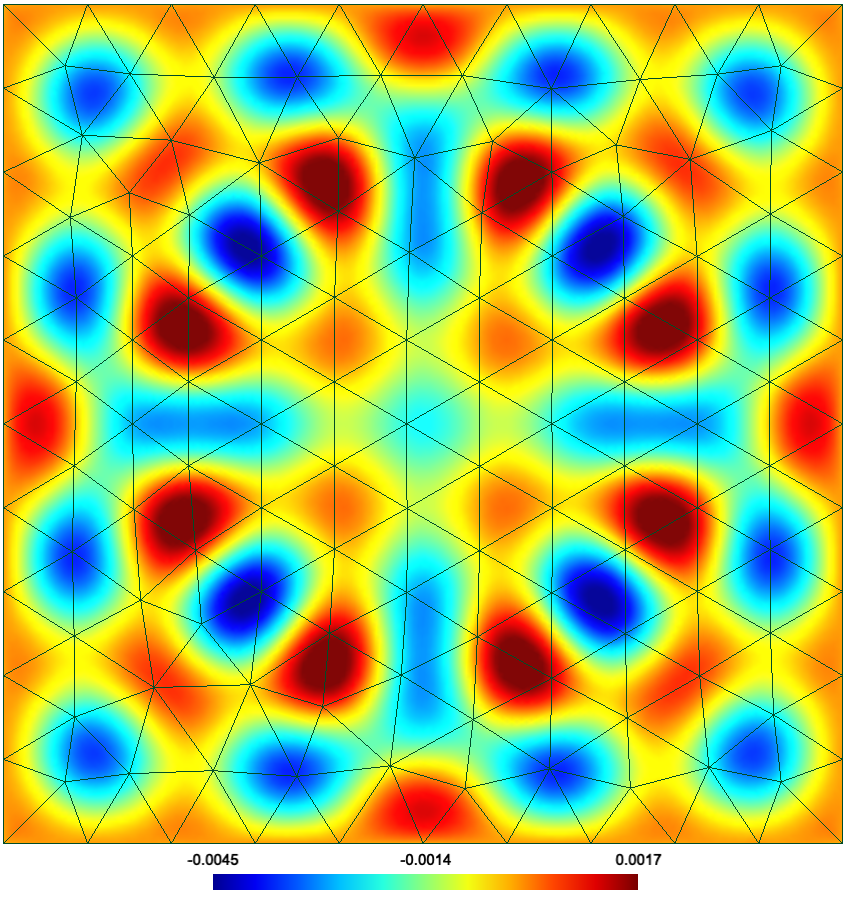}
\end{subfigure}
\bigskip
\begin{subfigure}[b]{0.9\textwidth}
  \centering
  \caption{Benchmark 3 (half-open waveguide)}
  \includegraphics[width=0.99\textwidth]{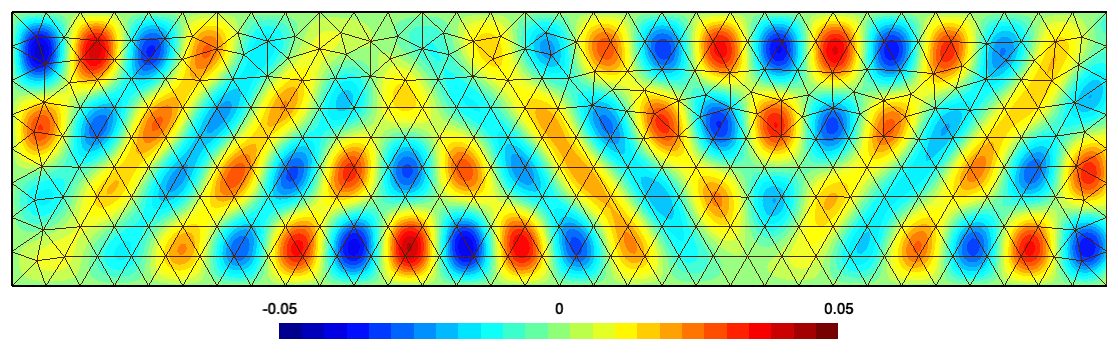}
\end{subfigure}
\caption{Snapshots of the real part of the solution for the three benchmarks with the default parameters.}
\label{fig:benchmarks}
\end{figure}

\paragraph{Benchmark 1 (Plane wave).}
The first benchmark is a simple plane wave propagating in the unit square domain $\Omega={]0,1[}\times{]0,1[}$.
The reference solution reads
\begin{align*}
  u_\mathrm{ref}(\bx) = e^{\im \kappa \bd\cdot\bx},
\end{align*}
with the propagation direction $\bd = (\cos\theta, \sin\theta)$ and a given angle $\theta$.
A non-homogeneous Robin condition is prescribed on the boundary of the domain
(i.e.~$\Gamma_{\mathrm{R}} := \partial\Omega$) with the appropriate right-hand side term.
By default, the parameters are $\kappa=15\pi$ and $h=1/16$.
We have also considered a wavenumber twice larger, $\kappa=30\pi$, with a spatial step $h=1/34$ corresponding to a relative error close to the one with the default parameters.

\paragraph{Benchmark 2 (Cavity).}
The second benchmark is a cavity problem.
The computational domain is again the unit square domain $\Omega={]0,1[}\times{]0,1[}$.
A homogeneous Dirichlet condition is prescribed on the boundary
of the domain (i.e.~$\Gamma_{\mathrm{D}} := \partial\Omega$), and a unit source term is
used in the Helmholtz equation:
\begin{align*}
\left\{
\begin{aligned}
  -\Delta u - \kappa^2 u &= 1, && \text{in $\Omega$}, \\
  u &= 0, && \text{on $\Gamma_{\mathrm{D}}$}. \\
\end{aligned}
\right.
\end{align*}
The reference solution is real. The eigenvalues and eigenmodes of this problem are
$\kappa^2_{n,m} := (n^2+m^2)\pi^2$ and $u_{n,m} := \sin(n\pi x_1)\sin(m\pi x_2)$,
respectively, for all $m,n > 0$. The reference solution is obtained semi-analytically
by truncating the Fourier expansion (see e.g.~\cite{chaumont2022controllability}).
By default, the parameters are $\kappa=(7+1/10)\sqrt{2}\pi$ and $h=1/10$.
We have also considered a wavenumber closer to an eigenvalue, $\kappa=(7+1/100)\sqrt{2}\pi$,
with a spatial step $h=1/15$ corresponding to a relative error close to the one with the
default parameters.

\paragraph{Benchmark 3 (Waveguide).}
The last benchmark is a half open waveguide problem.
The domain is $\Omega={]0,4[}\times{]0,1[}$, with a given length $L$.
The open side of the waveguide corresponds to the right side of $\Omega$.
An incident plane wave is prescribed at the open side by using a non-homogeneous Robin condition:
\begin{align*}
  \partial_n u - \im \kappa u &= e^{\im \kappa \bd\cdot\bx},
  \quad \text{on $\Gamma_{\mathrm{R}} := \{4\}\times{]0,1[}$},
\end{align*}
with the propagation direction $\bd = (\cos\theta, \sin\theta)$ and a given angle $\theta$.
A homogeneous Dirichlet condition is prescribed on the other sides of $\Omega$.
The reference solution is computed by using a semi-analytical approach described in
\cite{chaumont2022controllability}. By default, the parameters are $\kappa=6\pi$ and $h=1/8$.
We have also considered a wavenumber twice larger, $\kappa=12\pi$, with a spatial step $h=1/17$
corresponding to a relative error close to the one with the default parameters.


\subsection{Memory storage}
\label{sec:num:storage}

The total numbers of degrees of freedom (DOFs) with the DG, HDG and CHDG methods are given respectively by
\begin{align*}
  \#(\mathrm{dof}_\mathrm{DG}) &= 3 N_\mathrm{tri} N_\mathrm{dof\cdot per\cdot tri}, \\
  \#(\mathrm{dof}_\mathrm{HDG}) &= N_\mathrm{fce} N_\mathrm{dof\cdot per\cdot fce}, \\
  \#(\mathrm{dof}_\mathrm{CHDG}) &= 3 N_\mathrm{tri} N_\mathrm{dof\cdot per\cdot fce} = \left(N_\mathrm{fce\cdot bnd} + 2 N_\mathrm{fce\cdot int}\right) N_\mathrm{dof\cdot per\cdot fce},
\end{align*}
with the number of faces $N_\mathrm{fce}$, the number of boundary faces
$N_\mathrm{fce\cdot bnd}$, the number of interior faces $N_\mathrm{fce\cdot int}$
and the number of triangles $N_\mathrm{tri}$. Let us note that
$N_\mathrm{fce} = N_\mathrm{fce\cdot bnd} + N_\mathrm{fce\cdot int}$ and
$3 N_\mathrm{tri} = N_\mathrm{fce\cdot bnd} + 2 N_\mathrm{fce\cdot int}$.
For a scalar field, the numbers of DOFs per triangle and per face are given respectively by
$N_\mathrm{dof\cdot per\cdot tri} = (p+1)(p+2)/2$ and $N_\mathrm{dof\cdot per\cdot fce} = p+1$,
where $p$ is the polynomial degree.

The number of DOFs is obviously far smaller with the hybridizable methods.
It is nearly twice larger with CHDG than with HDG because there are two
characteristic variables per interior face and only one numerical trace.
The results would be similar in three dimensions.

Upper bounds for the numbers of non-zero elements in the global sparse matrix $\matalg{A}$
of the DG, HDG and CHDG systems are given respectively by
\begin{align*}
  \#(\mathrm{nnz}_\mathrm{DG}) &\lesssim N_\mathrm{tri} \left( 7 N_\mathrm{dof\cdot per\cdot tri}^2 + 54 N_\mathrm{dof\cdot per\cdot fce}^2 \right), \\
  \#(\mathrm{nnz}_\mathrm{HDG}) &\lesssim N_\mathrm{fce} \left( 5 N_\mathrm{dof\cdot per\cdot fce}^2 \right),  \\
  \#(\mathrm{nnz}_\mathrm{CHDG}) &\lesssim N_\mathrm{fce} \left( 8 N_\mathrm{dof\cdot per\cdot fce}^2 \right).
\end{align*}
For the hybridizable methods, the matrix $\matalg{A}$ is obtained after the elimination of the
physical unknowns. These bounds have been computed by using the rough approximation
$N_\mathrm{fce\cdot bnd}\ll N_\mathrm{fce\cdot int}$, which is valid only for
large benchmarks. Under this approximation, we have
\begin{align*}
  \frac{\#(\mathrm{nnz}_\mathrm{CHDG})}{\#(\mathrm{nnz}_\mathrm{HDG})} \approx 1.6.
\end{align*}
For the matrices of the reference benchmarks with the default parameters, this ratio varies
between $1.54$ and $1.66$ (see Table~\ref{tab:cond}). For three-dimensional problems with
tetrahedral elements, a similar reasoning leads to a ratio equal to $1.43$.
Therefore, although there are nearly twice as many DOFs with CHDG than with HDG,
the number of non-zero elements is not increased as much.

\begin{table}[tb!]
\caption{Number of degrees of freedom (\#dof) and number of non-zero entries (\#nnz) in $\matalg{A}$ for the different DG methods (i.e.~standard DG without hybridization, HDG and CHDG).}
\begin{center}\vspace{-1em}\small
\begin{tabular}{|c|c|rr|}
  \hline
  & & \#dof & \#nnz \\
  \hline
  & DG   & $18420$ & $734626$ \\
  Benchmark 1
  & HDG  &  $3812$ &  $74192$ \\
  & CHDG &  $7368$ & $109082$ \\
  \hline
  & DG   & $7260$ & $286715$ \\
  Benchmark 2
  & HDG  & $1532$ &  $27970$ \\
  & CHDG & $2904$ &  $43830$ \\
  \hline
  & DG   & $19260$ & $765006$ \\
  Benchmark 3
  & HDG  &  $4012$ &  $75149$ \\
  & CHDG &  $7704$ & $116202$ \\
  \hline
\end{tabular}
\end{center}
\label{tab:cond}
\end{table}


\subsection{Conditioning of the local matrices}
\label{sec:num:condLoc}

With the hybridizable approaches, the construction of the matrix $\matalg{A}$, and the application of $\matalg{A}$ in matrix-free iterative procedures, requires the solution of local element-wise algebraic systems.
For the HDG and CHDG methods, these systems correspond to Problems~\ref{pbm:HDG:local} and~\ref{pbm:CHDG:local}, respectively.
A bad conditioning of these systems could impact the quality of the numerical solution, regardless of the solution procedure.

As a preliminary study of the conditioning of the local systems, we first consider an elementary configuration used in~\cite{gopalakrishnan2015stabilization}.
The local systems are defined on a square element $K$ of side length $h$ with the lowest polynomial degree, i.e.~$p=0$.
With the HDG method, the local matrix corresponding to Problems~\ref{pbm:HDG:local} with the shape functions $\phi_1=1$, $\bphi_1=[1,0]^\top$ and $\bphi_2=[0,1]^\top$ reads
\begin{align*}
  \matalg{A}_\text{loc} = \mathrm{diag}(4h-\im\kappa h^2, -\im\kappa h^2, -\im\kappa h^2)
\end{align*}
and the condition number of this matrix is
\begin{align}
  \mathrm{cond}(\matalg{A}_\text{loc}) = \sqrt{1 + 16/(\kappa h)^2}.
  \label{eqn:condLoc:HDG}
\end{align}
With the CHDG method, the local matrix corresponding to Problem~\ref{pbm:CHDG:local} reads
\begin{align*}
  \matalg{A}_\text{loc} = \mathrm{diag}(2h-\im\kappa h^2, h-\im\kappa h^2, h-\im\kappa h^2)
\end{align*}
and the condition number of this matrix is
\begin{align}
  \mathrm{cond}(\matalg{A}_\text{loc}) = \sqrt{((\kappa h)^2 + 4)/((\kappa h)^2 + 1)}.
  \label{eqn:condLoc:CHDG}
\end{align}
The condition number of the HDG local matrix is always the largest.
In addition, this matrix becomes ill-conditioned as $\kappa h$ goes to zero, whereas the CHDG local matrix stays well-conditioned with $\mathrm{cond}(\matalg{A}_\text{loc})\approx 2$ for small values of $\kappa h$.
Although this simple setting is not representative of practical situations, it already highlights the influence of the variables used in the hybridization on the conditioning of the local matrices.

\begin{figure}[tb!]
  \centering
  \includegraphics[width=0.80\textwidth]{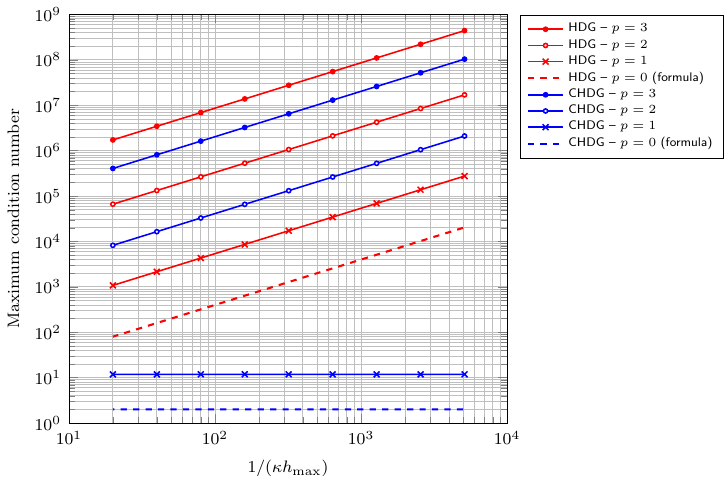}
  \caption{Maximum condition number of the local matrices with HDG (red curves) and CHDG (blue curves) as a function of $1/(\kappa h_{\max})$ for basis functions with polynomial degrees $p=1,2$ and $3$, where $h_{\max}$ is the length of the longest edge. The condition numbers corresponding to formulas \eqref{eqn:condLoc:HDG} and \eqref{eqn:condLoc:CHDG} are plotted with dashed lines.}
  \label{fig:num:condLoc}
\end{figure}

To continue the study, we consider a non-structured mesh for the unit square $\Omega={]0,1[^2}$.
This mesh is made of 1478 triangles and the length of the longest edge is close to $h_{\max}=0.05$.
The condition number of the corresponding local element-wise systems is computed for both HDG and CHDG, with different polynomial degrees $p=1,2,3$ and different wavenumbers $\kappa$.

The maximum condition number is plotted as a function of $1/(\kappa h_{\max})$ on Figure~\ref{fig:num:condLoc} for the different configurations.
The value $1/(\kappa h_{\max})$ is a measure of the mesh density in the denser region of the mesh.
We observe that the condition number increases linearly with $1/(\kappa h_{\max})$ in all the cases, except for CHDG with $p=1$.
Therefore, refining the mesh for a given wavenumber, or using a smaller wavenumber with a given mesh, increases the condition number of the local matrices.
Comparing the results with the HDG and CHDG methods for a given polynomial degree $p$, we observe that the condition number is always higher with HDG than with CHDG.
Increasing $p$ increases the condition number in all the cases.


\subsection{Conditioning of the global matrices}
\label{sec:num:condGlo}

The condition number of the global matrix $\matalg{A}$ is plotted a function of $1/(\kappa h_{\max})$ for the DG, HDG and CHDG methods on Figure~\ref{fig:numResu:condGlo}.
For each benchmark, two wavenumbers have been considered: the default wavenumber of the benchmark (denoted $\kappa_1$), and a second wavenumber corresponding to a more challenging case (denoted $\kappa_2$).
The second wavenumber is twice larger for benchmarks 1 and 3, and closer to a resonance mode for benchmark 2.
The condition number has been computed with the function \texttt{condest} in \textsf{MATLAB}.
For all the results, the relative error on the numerical solution is smaller than
$10^{-1}$. The black squares correspond to configurations with a relative error close to $10^{-2}$.

\begin{figure}[tb!]
  \centering
  \begin{subfigure}[b]{0.48\linewidth}
    \centering
    \caption{Benchmark 1 (plane wave) with $\left|\begin{array}{@{\:}l}\kappa_1=15\pi\\\kappa_2=30\pi\end{array}\right.$}
    \includegraphics[width=0.99\textwidth]{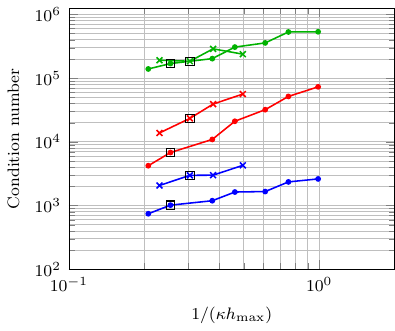}
  \end{subfigure}
  \hfill
  \begin{subfigure}[b]{0.48\linewidth}
    \centering
    \caption{Benchmark 2 (cavity) with $\left|\begin{array}{@{\:}l}\kappa_1=7.1\sqrt{2}\pi\\\kappa_2=7.01\sqrt{2}\pi\end{array}\right.$}
    \includegraphics[width=0.99\textwidth]{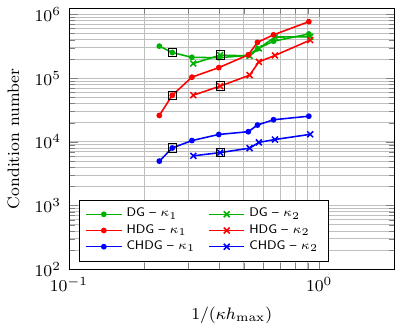}
  \end{subfigure}
  \\ \medskip
  \begin{subfigure}[b]{0.48\linewidth}
    \centering
    \caption{Benchmark 3 (waveguide) with $\left|\begin{array}{@{\:}l}\kappa_1=6\pi\\\kappa_2=12\pi\end{array}\right.$}
    \includegraphics[width=0.99\textwidth]{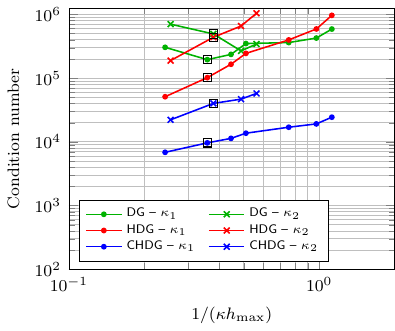}
  \end{subfigure}
  \caption{Condition number of $\matalg{A}$, the matrix of the physical system (for DG) or the matrix of the reduced system (for HDG and CHDG), as a function of $1/(\kappa h_{\max})$ for the three benchmarks, where $h_{\max}$ is the length of the longest edge.
  For each benchmark, two wavenumbers are considered, $\kappa_1$ and $\kappa_2$.
  The black squares correspond to configurations with a relative error close to $10^{-2}$.}
  \label{fig:numResu:condGlo}
\end{figure}

We observe on Figure~\ref{fig:numResu:condGlo} that the condition number is always smaller with CHDG than with HDG and DG by one or two orders of magnitude in nearly all the cases.
Moreover, the condition number increases nearly linearly with $1/(\kappa h_{\max})$ for DG and CHDG, while the increase is nearly quadratic for HDG.

The influence of $\kappa$ on the condition number is similar for HDG and CHDG.
Indeed, for each benchmark, the condition number is larger with the larger wavenumber.
By contrast, the condition number for the DG method without hybridization does not vary much with $\kappa$.


\section{Iterative solution procedures}
\label{sec:iterProc}

In this section, we study the efficiency of iterative procedures for solving the linear systems
resulting from the DG discretization and the two hybridization strategies. With the
CHDG approach, the fixed-point iterative procedure can be considered thanks to
the specific structure of the global matrix, which we analyzed in Section~\ref{sec:redSys}.
The convergence of the fixed-point iterative scheme with CHDG is discussed in
Section~\ref{sec:iter:Richardson}. The performance of DG, HDG and CHDG with
standard iterative schemes is discussed in Section~\ref{sec:iter:comp}.


\subsection{Convergence of the fixed-point iterative scheme for CHDG}
\label{sec:iter:Richardson}

We consider the algebraic system obtained by using the CHDG approach (Problem~\ref{pbm:CHDG:hybridized}) with the discretization described in Section~\ref{sec:num:discretization}.
This CHDG system can be written as
\begin{align*}
  (\matalg{I} - \matalg{\Pi}\matalg{S})\vecalg{g} = \vecalg{b},
\end{align*}
where $\matalg{I}$, $\matalg{\Pi}$ and $\matalg{S}$ are the identity, exchange and scattering matrices, respectively.
As the operator $\Pi\TS$ is a strict contraction (Corollary~\ref{corol:contraction}), the spectral radius of $\matalg{\Pi}\matalg{S}$ is strictly lower than $1$, i.e.~$\rho(\matalg{\Pi}\matalg{S}) < 1$.
Therefore, the Richardson iterative scheme applied to this system shall converge without relaxation (see e.g.~\cite{saad2003iterative}).
For a given initial guess $\vecalg{g}^{(0)}$, the procedure reads
\begin{align*}
  \vecalg{g}^{(\ell+1)} &= \matalg{\Pi} \matalg{S} \vecalg{g}^{(\ell)} + \vecalg{b}, \qquad \text{for $\ell=0,1,\dots$}
\end{align*}
If the eigenvalues of the iteration operator are far from the unit disk,
this procedure will converge rapidly.
As discussed in Section \ref{sec:fixedpointpbm}, this will depend on both the dissipative properties
of the upwind DG scheme and on the physical dissipation in the problem under consideration.

As a preliminary verification, we discuss the eigenvalues of the iteration matrix
$\matalg{\Pi}\matalg{S}$ and the spectral radius $\rho(\matalg{\Pi}\matalg{S})$
by using the numerical benchmarks. The eigenvalues of the iteration matrix are
represented on Figure~\ref{fig:numResu:spectrum} for the three benchmarks with
the default parameters. The values of $1-\rho(\matalg{\Pi}\matalg{S})$ are given
in Table~\ref{tab:num:convRate} for different sets of parameters.
The eigenvalues and the spectral radius are obtained by using the function
\texttt{eigs} in \textsf{MATLAB}.

\begin{figure}[tb!]
  \centering
  \begin{subfigure}[b]{0.48\linewidth}
    \centering
    \caption{Benchmark 1 (plane wave)}
    \includegraphics[width=0.99\textwidth]{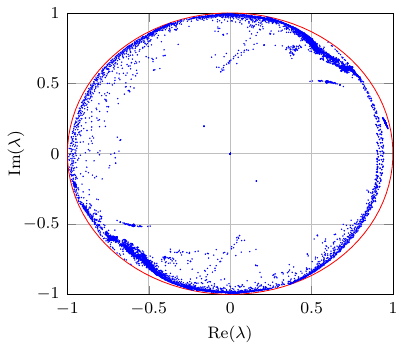}
  \end{subfigure}
  \hfill
  \begin{subfigure}[b]{0.48\linewidth}
    \centering
    \caption{Benchmark 2 (cavity)}
    \includegraphics[width=0.99\textwidth]{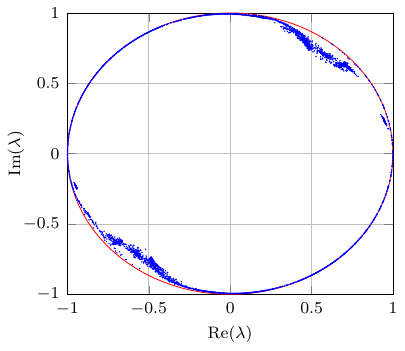}
  \end{subfigure}
  \\ \medskip
  \begin{subfigure}[b]{0.48\linewidth}
    \centering
    \caption{Benchmark 3 (waveguide)}
    \includegraphics[width=0.99\textwidth]{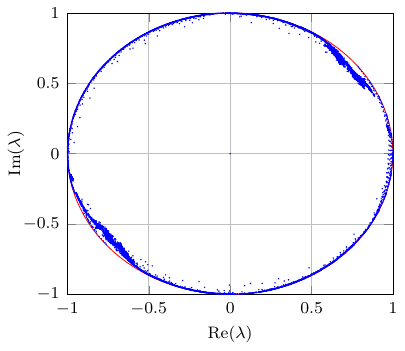}
  \end{subfigure}
  \caption{Spectrum of the iteration matrix $\matalg{\Pi}\matalg{S}$ of the fixed-point iterative scheme for the three benchmarks with the default parameters and the CHDG method. The unit circle is plotted in red.}
  \label{fig:numResu:spectrum}
\end{figure}

\begin{table}[!tb]
\caption{Spectral radius $\rho$ of the iteration matrix $\matalg{\Pi}\matalg{S}$ of the fixed-point iterative scheme for the three benchmarks with different parameters and the CHDG method.}
\begin{center}\vspace{-1em}\small
\begin{tabular}{|c|c@{\;\;}c@{\;\;}c|c@{\;\;}c@{\;\;}c|c@{\;\;}c@{\;\;}c|}
  \hline
  & \multicolumn{3}{c|}{Benchmark 1 (plane wave)} & \multicolumn{3}{c|}{Benchmark 2 (cavity)} & \multicolumn{3}{c|}{Benchmark 3 (waveguide)} \\
  \hline
  $\kappa$
    & $15\pi$ & $15\pi$ & $30\pi$
    & $7.1\sqrt{2}\pi$ & $7.1\sqrt{2}\pi$ & $7.01\sqrt{2}\pi$
    & $6\pi$ & $6\pi$ & $12\pi$ \\
  $h$
    & $1/16$ & $1/34$ & $1/34$
    & $1/10$ & $1/15$ & $1/15$
    & $1/8$ & $1/17$ & $1/17$ \\
  $\kappa h$
    & $2.95$ & $1.39$ & $2.77$
    & $3.15$ & $2.10$ & $2.08$
    & $2.36$ & $1.11$ & $2.22$ \\
  \hline
  {\scriptsize$1-\rho(\matalg{\Pi}\matalg{S})$}
    & $2.9\:10^{-3}$ & $7.8\:10^{-5}$ & $5.5\:10^{-4}$
    & $2.8\:10^{-4}$ & $1.5\:10^{-5}$ & $1.4\:10^{-5}$
    & $5.5\:10^{-5}$ & $2.5\:10^{-6}$ & $2.9\:10^{-5}$ \\
  \hline
\end{tabular}
\end{center}
\label{tab:num:convRate}
\end{table}

In all the cases, the eigenvalues are strictly inside the unit circle,
which is in agreement with the theoretical result.
We shall also observe in the next section that the iterative process effectively converges.
Nevertheless, some eigenvalues are close to the unit circle, so that the spectral radius
is close to one. For every benchmark, we observe that the spectral radius is
closer to one when using a finer mesh (second column of each benchmark in Table~\ref{tab:num:convRate})
or when using the second wavenumber with the fine mesh (third column).


\subsection{Comparison of DG, HDG and CHDG with standard iterative schemes}
\label{sec:iter:comp}

In practice, the iterative procedures to solve large-scale time-harmonic problems can be rather sophisticated, because the corresponding algebraic linear systems are generally non-Hermitian and ill-conditioned.
The GMRES \emph{(generalized minimal residual)} method with restart and preconditioning strategies
is one of the most widely used approach. For the standard version without restart, the
convergence is guaranteed, but the computational cost increases with the number iterations,
both in terms of memory storage and floating-point operations. Alternative Krylov methods are
frequently considered, with smaller computational cost per iteration and smaller memory
footprint, but at the price of a larger number of iterations and/or a convergence that is
not always guaranteed.

For the sake of brevity, we only consider three standard iterative schemes
to compare the DG methods: the fixed-point iterative scheme (for CHDG only),
the GMRES iteration
without restart and the CGNR \emph{(conjugate gradient normal)} method.
The CGNR iteration
corresponds to the conjugate gradient method applied to the normal equation $\matalg{A}^*\matalg{A} \vecalg{g}=\matalg{A}^*\vecalg{b}$.
For a given initial solution $\vecalg{g}^{(0)}$, both GMRES and CGNR produce an approximate
solution $\vecalg{g}^{(\ell)}$ at step $\ell$ that belongs to a certain Krylov subspace and
that minimizes the $2$-norm of the residual, i.e.
$\vecalg{g}^{(\ell)}$ minimizes $f(\vecalg{g}) = \|\vecalg{b}-\matalg{A}\vecalg{g}\|_2$.
The approximate solution belongs to
$\vecalg{g}^{(0)} + \mathcal{K}_\ell(\matalg{A},\vecalg{r}^{(0)})$
with GMRES and to
$\vecalg{g}^{(0)} + \mathcal{K}_\ell(\matalg{A}^*\matalg{A},\matalg{A}^*\vecalg{r}^{(0)})$
with CGNR, where $\mathcal{K}_\ell$ is the Krylov subspace of order $\ell$
(see e.g.~\cite{saad2003iterative}).
The convergence rate of the CGNR iterative process depends on the condition number of $\matalg{A}$.
The convergence can be slow if the condition number is large.
Nevertheless, we have observed that the condition number is nearly always smaller with CHDG
than with the other approaches (see Section~\ref{sec:num:condGlo}).

To study the efficiency of the iterative schemes with the different methods,
we consider the relative error of the physical fields defined as
\begin{align*}
  \sqrt{\frac{\left\|u_h - u_\mathrm{ref}\right\|_\Omega^2 + \left\|\bq_h - \bq_\mathrm{ref}\right\|_\Omega^2}{\left\|u_\mathrm{ref}\right\|_\Omega^2 + \left\|\bq_\mathrm{ref}\right\|_\Omega^2}},
\end{align*}
where $u_\mathrm{ref}$ and $\bq_\mathrm{ref}$ correspond to the reference analytical or
semi-analytical solution.
The history of relative error is plotted in Figure~\ref{fig:num:histo} for CGNR (lines with marker $\circ$), GMRES (lines with marker $\bullet$) and the fixed-point
iteration in the CHDG case (lines with marker $\times$). The results have been obtained
for DG without hybridization (green lines), HDG (red lines) and CHDG (blue lines).
The relative error obtained with a direct solver is indicated by the horizontal dashed line.

\begin{figure}[p]
  \centering
  \includegraphics[width=0.90\textwidth]{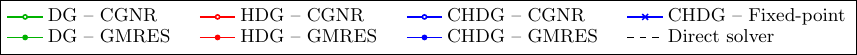}
  \\ \medskip
  \begin{subfigure}[b]{0.48\linewidth}
    \centering
    \caption{Benchmark 1 -- $\kappa=15\pi$ -- $h=1/16$}
    \label{fig:num:histo:1A}
    \includegraphics[width=0.99\textwidth]{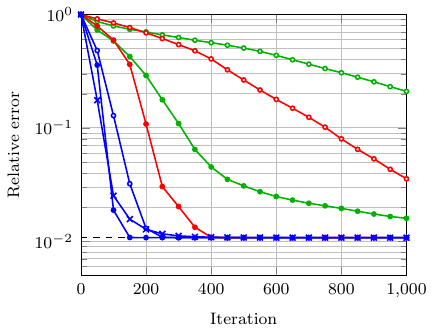}
  \end{subfigure}
  \hfill
  \begin{subfigure}[b]{0.48\linewidth}
    \centering
    \caption{Benchmark 1 -- $\kappa=30\pi$ -- $h=1/34$}
    \label{fig:num:histo:1B}
    \includegraphics[width=0.99\textwidth]{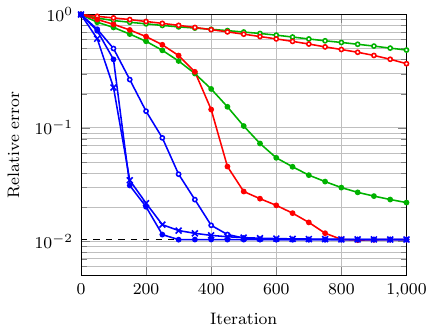}
  \end{subfigure}
  \\ \medskip
  \begin{subfigure}[b]{0.48\linewidth}
    \centering
    \caption{Benchmark 2 -- $\kappa=(7+1/10)\sqrt{2}\pi$ -- $h=1/10$}
    \label{fig:num:histo:2A}
    \includegraphics[width=0.99\textwidth]{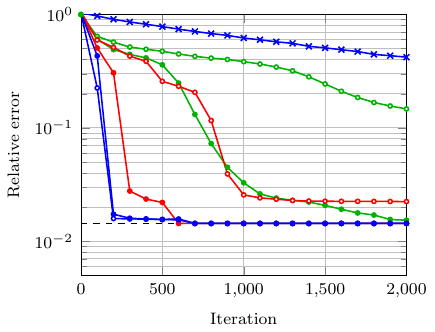}
  \end{subfigure}
  \hfill
  \begin{subfigure}[b]{0.48\linewidth}
    \centering
    \caption{Benchmark 2 -- $\kappa=(7+1/100)\sqrt{2}\pi$ -- $h=1/15$}
    \label{fig:num:histo:2B}
    \includegraphics[width=0.99\textwidth]{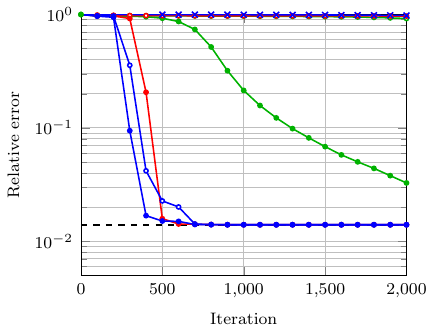}
  \end{subfigure}
  \\ \medskip
  \begin{subfigure}[b]{0.48\linewidth}
    \centering
    \caption{Benchmark 3 -- $\kappa=6\pi$ -- $h=1/8$}
    \label{fig:num:histo:3A}
    \includegraphics[width=0.99\textwidth]{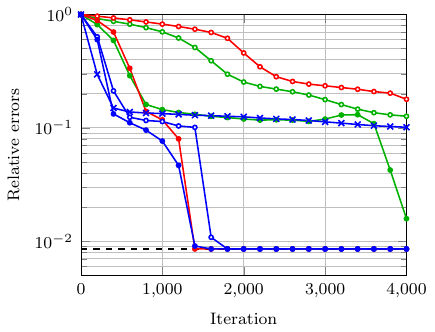}
  \end{subfigure}
  \hfill
  \begin{subfigure}[b]{0.48\linewidth}
    \centering
    \caption{Benchmark 3 -- $\kappa=12\pi$ -- $h=1/17$}
    \label{fig:num:histo:3B}
    \includegraphics[width=0.99\textwidth]{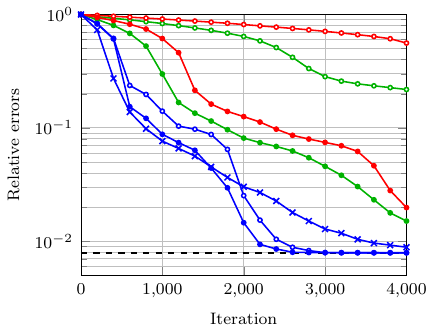}
  \end{subfigure}
  \\ \medskip
  \caption{Error history for the three benchmarks with different iterative schemes and different DG schemes.
  The dashed horizontal lines correspond to the relative errors obtained with a direct solver.}
  \label{fig:num:histo}
\end{figure}

First, let us analyze the results obtained with CHDG and fixed-point iterations
(blue lines with marker $\times$). The following observations can be made:
\begin{itemize}
  \item For benchmark 1 (plane wave), the convergence of the iterative process is very fast.
  The decay of error is slightly slower with the higher wavenumber.
  Compared to the other approaches, CHDG with fixed-point iterations provides nearly the fastest convergence.
  \item By contrast, for benchmark 2 (cavity), the convergence of the fixed-point iterations is very slow.
  This can be explained by the fact that this benchmark does not feature any physical absorption.
  Therefore, as discussed in Section \ref{sec:fixedpointpbm}, the only source of dissipation comes
  from the DG scheme.
  The decay of error is much slower for the wavenumber closer to the resonance.
  Compared to the other methods, this approach provides the slowest convergence.
  \item For benchmark 3 (half-open waveguide) with the first set of parameters (Figure~\ref{fig:num:histo:3A}), the relative error decays relatively rapidly during the 500 first iterations, then the decay slows down dramatically, and the relative error is only about $10^{-1}$ at iteration 4,000.
  With the higher wavenumber (Figure~\ref{fig:num:histo:3B}), the relative error decays more rapidly until approximately $10^{-2}$ at iteration 4,000.
  \item The asymptotic regime of convergence has been reached in three cases, and the slopes of error decay are coherent to the spectral radii obtained in Table~\ref{tab:num:convRate}: $\rho = 1 - 2.8\:10^{-4}$ for Figure~\ref{fig:num:histo:2A}, $\rho = 1 - 1.5\:10^{-5}$ for Figure~\ref{fig:num:histo:2B}, and $\rho = 1 - 5.5\:10^{-5}$ for Figure~\ref{fig:num:histo:3A}.
  The asymptotic regime starts at the beginning of the iterations in the cavity case.
\end{itemize}
To summarize, the fixed-point iterative process effectively converges for CHDG, but the performance strongly depends on the physical setting.
The convergence can be very fast for purely propagating cases, and very slow for cavity or waveguide cases.
In the latter cases, the asymptotic regime, which can start relatively quickly, is rather slow.

We then discuss
the convergence of the CGNR and GMRES schemes with the different approaches,
i.e.~DG without hybridization, HDG and CHDG.
We can make the following comments:
\begin{itemize}
\item
When using CGNR (lines with marker $\circ$ on Figure~\ref{fig:num:histo}),
the convergence is much faster with CHDG than with HDG and DG without hybridization
in all the cases. Comparing the last two approaches, the convergence is faster with
HDG than with DG without hybridization on Figures~\ref{fig:num:histo:1A}, \ref{fig:num:histo:1B}
and \ref{fig:num:histo:2A}, and the converse is true on Figures~\ref{fig:num:histo:2B},
\ref{fig:num:histo:3A} and \ref{fig:num:histo:3B}.
\item
When using GMRES (lines with marker $\bullet$ on Figure~\ref{fig:num:histo}), the fastest convergence is still obtained with CHDG in all the cases, but the convergence is rather close with HDG for the cavity benchmark (Figures~\ref{fig:num:histo:2A}-\ref{fig:num:histo:2B}) and the first waveguide benchmark (Figures~\ref{fig:num:histo:3A}).
The convergence is generally faster with HDG than with DG without hybridization, but the converse is true for the second waveguide benchmark (Figures~\ref{fig:num:histo:3B}).
\end{itemize}
To summarize, if the problem is solved with either CGNR or GMRES, the convergence of the iterative process is always faster with the CHDG method.
Using the standard HDG method generally speeds up the convergence in comparison with the DG method without hybridization, but the converse is true for several cases.

Finally, let us compare the performance of CGNR and GMRES when the CHDG method is used
(blue lines with markers $\circ$ and $\bullet$ on Figure~\ref{fig:num:histo}).
The convergence is always slightly faster with GMRES than with CGNR, but the difference
is not very large. In the worst case (Figure~\ref{fig:num:histo:1B}), the number of
iterations to achieve the reference relative error (obtained with the direct solver) is
twice larger with CGNR than with GMRES. Considering the computational cost of GMRES, which
increases at each iteration, the CGNR is a potential good candidate for realistic cases.
The complete analysis of the runtimes and computational costs, which depend on the implementation,
will be performed in future works


\section{Conclusion}
\label{sec:conclu}

In this work, we propose a new hybridization technique, which we call the CHDG method,
for solving time-harmonic problems with upwind DG discretizations.
The auxiliary unknowns used in the CHDG method correspond to characteristic variables,
whereas the auxiliary unknowns used in the standard approach correspond to a Dirichlet trace.
At the price of increasing the required memory storage for the reduced linear system,
this choice largely improves its properties
and makes it more suitable for iterative solution procedures.

We study the properties of the local element-wise problems
and the global linear systems for the standard HDG method and the CHDG method.
In order to investigate how the original DG scheme and its hybridized versions interplay
with usual iterative solvers, we provide a set of 2D numerical results
where the auxiliary unknowns are discretized with scaled Legendre basis functions.
The key properties of the CHDG may be summarized as follows.

With CHDG, the reduced system can be written in the form $(\TI-\Pi\TS)g=b$,
where the operator $\Pi\TS$ is a strict contraction.
It can be solved with a fixed-point iteration without relaxation.
This fixed-point iteration converges quickly in open domains, but unfortunately,
the convergence becomes slow when waves are trapped, like in waveguides or cavities.

The memory storage required to store an unknown vector of the reduced system
is twice larger with CHDG than with the standard HDG method.
Similarly, the number of non-zero entries in the CHDG matrix is multiplied by about
1.6 in 2D and 1.4 in 3D as compared to the HDG matrix, with a similar filling pattern.
In return, the condition number of the matrices of the local element-wise systems
is always smaller with CHDG than with the standard HDG method.
Similarly, the condition number of the global reduced matrix
is also always smaller with CHDG than with HDG.
It is also smaller than the condition number of the global matrix
of the DG system without hybridization.

For the iterative solution procedure, we have employed the usual GMRES iteration
(without restart) and the CGNR iteration.
In both cases, the convergence of the iterative process is always faster with CHDG
than with HDG or DG without hybridization.
Focusing on the CHDG system, the number of CGNR iterations is always larger
than the number of GMRES iterations, but the difference is rather limited for the
benchmarks considered in this article. Since restart must be employed for GMRES in practice,
and since each GMRES iteration is typically more costly than the corresponding CGNR iteration,
we believe that CGNR may be a competitive approach to solve the CHDG system.

Although we focus on 2D benchmarks here, the definition of the method is valid for 3D cases.
Besides, the method is in principle not restricted to scalar problems, and
electromagnetic or elastic waves should be accessible as well because similar DG
schemes with upwind fluxes are already available. In future works, we will
investigate in more depth the computational aspects for solving iteratively 3D cases,
high-order transmission conditions, and combinations with preconditioning techniques
and domain decomposition methods to accelerate further the convergence of the iterative
solution procedures.


\paragraph{Acknowledgments.}
This work was supported in part by the ANR JCJC project \emph{WavesDG} (research grant ANR-21-CE46-0010).
The authors thank X.~Antoine, H.~B\'eriot, X.~Claeys, G.~Gabard, C.~Geuzaine and S.~Pescuma for helpful discussions and remarks.

\small
\setlength{\bibsep}{0pt plus 0ex}
\bibliographystyle{abbrvnat}
\bibliography{myrefs}
\addcontentsline{toc}{section}{References}

\end{document}